\documentclass{amsart}
\usepackage{amsmath, amssymb, amsthm, amsfonts,
   enumerate,
  hyperref,verbatim}

\usepackage{graphicx}

\usepackage{thmtools}

\theoremstyle{plain}
\newtheorem{thm}{Theorem}
\numberwithin{thm}{section}
\newtheorem{prp}[thm]{Proposition}

\newtheorem{crl}[thm]{Corollary}

\theoremstyle{definition}
\newtheorem{dfn}[thm]{Definition}

\newtheorem{rmk}[thm]{Remark}
\numberwithin{equation}{section}

\hypersetup{
    bookmarks=true,         
    unicode=false,          
    pdftoolbar=true,        
    pdfmenubar=true,        
    pdffitwindow=false,     
    pdfstartview={FitH},    
    pdftitle={My title},    
    pdfauthor={Author},     
    pdfsubject={Subject},   
    pdfcreator={Creator},   
    pdfproducer={Producer}, 
    pdfkeywords={keyword1} {key2} {key3}, 
    pdfnewwindow=true,      
    colorlinks=true,       
    linkcolor=black,          
    citecolor=magenta,        
    filecolor=magenta,      
    urlcolor=cyan           
}
\def\calS{\mathcal{S}}

\newcommand{\Id}{\textup{Id}}

\newcommand{\lap}{\ensuremath{\Delta}}

\newcommand\ep{\epsilon}

\newcommand{\pa}				{\ensuremath{\partial}}

\newcommand{\eq}				{\begin{equation}}
\newcommand{\eqq}				{\end{equation}}

\newcommand{\hw}
			{\sqrt{- \lap}}

\newcommand{\oc}{\mathrm{1c}}
\newcommand{\ocul}{\oc}
\newcommand{\bl}{{\mathrm b}}
\newcommand{\scl}{{\mathrm{sc}}}
\newcommand{\cul}{{\mathrm{cu}}}
\newcommand{\para}{{\mathrm{para}}}

\newcommand{\scH}{{}^\scl H}

\newcommand{\CI}{C^\infty}
\newcommand{\dCI}{\dot C^\infty}

\newcommand{\Tsc}{{}^{\scl}T}

\newcommand{\cF}{\mathcal F}
\newcommand{\cS}{\mathcal S}
\newcommand{\cL}{\mathcal L}

\newcommand{\RR}{\mathbb{R}}

\newcommand{\NN}{\mathbb{N}}

\newcommand{\sphere}{\mathbb{S}}
\newcommand{\supp}{\operatorname{supp}}

\newcommand{\Psisc}{\Psi_\scl}
\newcommand{\Psicu}{\Psi_\cul}
\newcommand{\Psiocu}{\Psi_\ocul}
\newcommand{\Psiocucu}{\Psi_{\ocul,\cul}}

\newcommand{\Psischh}{\Psi_{\scl,\semi,\cF}}

\newcommand{\Psiocuhh}{\Psi_{\ocul,\semi,\cF}}
\newcommand{\Psiscocuhh}{\Psi_{\scl,\ocul,\semi,\cF}}
\newcommand{\semi}{\hbar}
\newcommand{\Psipara}{\Psi_\para}

\newcommand{\Vb}{{\mathcal V}_{\bl}}
\newcommand{\Vcu}{{\mathcal V}_{\cul}}
\newcommand{\Vsc}{{\mathcal V}_{\scl}}

\newcommand{\Vocu}{{\mathcal V}_{\ocul}}
\newcommand{\Vocuh}{{\mathcal V}_{\ocul,\semi}}
\newcommand{\Vocuhh}{{\mathcal V}_{\ocul,\semi,\cF}}

\newcommand{\Hsc}{H_{\scl}}

\newcommand{\Hcu}{H_{\cul}}
\newcommand{\Hocu}{H_{\ocul}}
\newcommand{\Hocuh}{H_{\ocul,\semi}}
\newcommand{\Hocuhh}{H_{\ocul,\semi,\cF}}

\newcommand{\xisc}{\xi_{\scl}}
\newcommand{\etasc}{\eta_{\scl}}
\newcommand{\xiocu}{\xi_{\ocul}}
\newcommand{\etaocu}{\eta_{\ocul}}
\newcommand{\xicu}{\xi_{\cul}}
\newcommand{\etacu}{\eta_{\cul}}

\title{The X-ray transform on asymptotically conic spaces}

\author{Andr\'as Vasy and Evangelie Zachos}

\thanks{The authors gratefully acknowledge support from the National
  Science Foundation under grant number  DMS-1664683 and DMS-1953987.}

\date{\today}

\address{Department of Mathematics, Stanford University, Stanford, CA
94305-2125, U.S.A.}
\email{andras@math.stanford.edu}
\email{ezachos@alumni.stanford.edu}
\subjclass{53C65, 35S05}

\begin{document}
\begin{abstract}
In this paper, partly based on Zachos' PhD thesis, we show that the
geodesic X-ray transform is stably invertible near infinity on a class of
asymptotically conic manifolds which includes perturbations of
Euclidean space. In particular certain kinds of conjugate points are
allowed. Further, under a global convex foliation condition, the
transform is globally invertible.

The key analytic tool, beyond the approach introduced by Uhlmann and Vasy, is the introduction of a new pseudodifferential
operator algebra, which we name the 1-cusp algebra, and its
semiclassical version.
  \end{abstract}

\maketitle

\section{Introduction}

The geodesic X-ray transform $I$ on a 
Riemannian manifold (often with boundary) $(M,g)$, of dimension $n\geq
2$, is a map from a class of functions, such as continuous functions
on $M$, to a corresponding class functions on the space of geodesics:
if $\gamma$ is a geodesic, then
$$
(If)(\gamma) = \int f(\gamma(s))\, ds.
$$
Here one needs to make some assumption on the geometry and the
function $f$ so that the integral makes sense, for instance
ensuring that one integrates over a finite interval or that $f$ decays
sufficiently fast along geodesics.

An important and well-studied question is whether the X-ray transform
is (left-)
invertible. In other words, if $If$ is known, can $f$ be determined?
The answer, as one might expect, depends on $(M,g)$ and the class of
$f$ to be considered. In addition, this problem, or its tensorial version, is the linearization
of the boundary rigidity problem which asks whether
the restriction of the distance function $d_g$ to $\pa M\times\pa M$
determines $g$ up to diffeomorphisms, or if $g$ is in a fixed
conformal class, $g=c^{-2}g_0$, with $g_0$ fixed, whether the same
information determines the conformal factor $c$. (There are also some slightly different
versions of these questions with some additional data.)

A version of this problem was studied already
over a century ago by Herglotz \cite{Herglotz} and Wiechert and
Zoeppritz \cite{Wiechert1907}, in the special case when $M$ is a ball
with a rotationally symmetric metric $g$ and $f$ is also rotationally
invariant. These assumptions make the problem effectively
one-dimensional, yet there is actual geometry involved: they proved
the injectivity of $I$ under an additional assumption which is the
special case of the convex foliation assumption described
below.

The `standard' conjecture in the field is Michel's, namely that
boundary rigidity holds on simple manifolds
\cite{MR636880}. Recall that 
a Riemannian manifold with boundary $(M, g)$ is simple if for any $p \in M$, the
exponential map $\exp_p$ is a diffeomorphism from a neighborhood
of the origin of $T_pM$ and if $\pa M$ is
strictly convex with respect to $g$. There has been much work on this
problem, primarily on compact manifolds. As we shall see below, there
is a significant difference in the two vs.\ higher dimensional cases.
Croke and Otal independently established boundary rigidity in the
two-dimensional non-positively curved case \cite{MR1057248}, \cite{MR1036134},
before Pestov and Uhlmann proved Michel's conjecture in general in two-dimensions \cite{MR2153407}.
In the higher-dimensional setting, 
Stefanov and Uhlmann showed rigidity for metrics close to Euclidean
ones \cite{MR1618347}. Mukhometov showed rigidity for metrics
conformal to the Euclidean metric \cite{MR621466}. There are also
results under symmetry assumptions, while in
\cite{Lassas-Sharafutdinov-Uhlmann:Semiglobal} and
\cite{Burago-Ivanov:Boundary} boundary rigidity is shown when one of
the metrics is close to the Euclidean one, while
\cite{Stefanov-Uhlmann:Boundary} proves a generic result. Newer
developments will be described below.

As a motivation for the current work we recall a result
of Uhlmann and Vasy \cite{Uhlmann-Vasy:X-ray} concerning the local
X-ray transform which introduced a new approach to this inverse
problem. For an open set $O$ in a manifold with boundary, the local transform is the X-ray
transform restricted geodesic segments which are completely in $O$ with endpoints on $\pa M$.

\begin{thm}[Uhlmann and Vasy \cite{Uhlmann-Vasy:X-ray}]\label{thm:U-V:X-ray}
 For compact Riemannian manifolds $(M,g)$ with strictly convex
 boundary, the local geodesic X-ray transform is left-invertible on small
 enough collar neighborhoods of the boundary, and is globally
 left-invertible under a convex foliation assumption.	
\end{thm}

Here the convex foliation assumption is a replacement for the
simplicity condition; at this point the precise relationship between
these is not completely clear. We recall the precise definition: the
assumption is the existence of a $\CI$ function $x$ with
non-vanishing differential which is strictly concave from the side of
the super-level sets, i.e.\ for all geodesics $\gamma$,
$$
\frac{d(x\circ\gamma)}{ds}(s_0)=0\Rightarrow \frac{d^2(x\circ\gamma)}{ds^2}(s_0)>0.
$$
This assumption is satisfied, for instance, on domains in simply connected
negatively (non-positively) curved manifolds, with $x$ being the distance from a point
outside the domain, as well as on manifolds without conic points. A
simple modification of the proof allows some singular level sets, like
the radius function from the center of a ball, and then manifolds with
non-negative curvature are also covered, as shown in
\cite{Paternain-Salo-Uhlmann-Zhou:Matrix}. Indeed, the setting of Herglotz \cite{Herglotz} and Wiechert and
Zoeppritz \cite{Wiechert1907} becomes a special case of this setup.

The approach to this theorem was by adding an artificial boundary to
create a collar neighborhood of the actual boundary and showing that
the local geodesic X-ray
transform on this collar neighborhood was an invertible operator in a
particular operator class defined via microlocal analysis, as we explain below. These two authors, along with Stefanov, used this linear result to prove a nonlinear result about metric rigidity:

\begin{thm}[Stefanov, Uhlmann and Vasy \cite{Stefanov-Uhlmann-Vasy:Rigidity-Normal}, see also
  \cite{MR3454376} in the conformal case]
If $(M,g)$ is an $n$-dimensional Riemannian manifold with boundary,
where $n \ge 3$, with strictly convex boundary and a convex foliation,
then if there is another Riemannian metric $\hat g$ on $M$ such that
$\partial M$ is still strictly convex with respect to $\hat g$, and if
$g$ and $\hat g$ have identical boundary distance functions, then they
are the same up to a diffeomorphism fixing the boundary pointwise. 
\end{thm}

In this paper we extend the first result, Theorem~\ref{thm:U-V:X-ray}, to a class of asymptotically
conic manifolds. Recall that a conic metric, on a manifold
$(0,\infty)_r\times Y$, with $Y$ the cross section or link, which we
always assume is compact and without boundary, is one of the form
$$
g_\infty=dr^2+r^2 h,
$$
where $h$ is a Riemannian metric on $Y$. An asymptotically conic
metric is one on a manifold which outside a compact set is identified
with $(r_0,\infty)_r\times Y$, with a metric that on this conic end
tends to $g_\infty$ as $r\to\infty$ in a specified way. To be
concrete, for our purposes, it is useful to `bring in' infinity, i.e.\
let $x=r^{-1}$, so $r\to\infty$ corresponds to $x\to 0$, and add a
boundary $\{0\}_x\times Y$ to the manifold, thus compactifying it to
$\overline{M}$. An
asymptotically conic metric then, as introduced by Melrose
\cite{RBMSpec}, is a Riemannian metric on $M$ which is of the form
$$
g=\frac{dx^2}{x^4}+\frac{h}{x^2}
$$
near $\pa\overline{M}$, where $h$ is a smooth symmetric 2-cotensor on
$\overline{M}$; $g$ is thus asymptotic to $g_\infty$ given by
$h|_{x=0}$ on the cross section $Y$.

\begin{thm}[See Theorem~\ref{thm:main-conic-ops}
  and Corollary~\ref{cor:main-conic}]\label{thm:main-conic}
Suppose that $M$ is a manifold of dimension $\geq 3$, $g$ is an
asymptotically conic metric on $M$ for which the cone's
cross section (link) has no conjugate points within distance $\leq\pi/2$. Then on a collar
neighborhood of infinity the geodesic X-ray transform is injective on
the restriction to the collar neighborhood of
sufficiently rapidly decaying
exponential-power type weighted function spaces, i.e.\ for all $p>0$
there is $C>0$ such that injectivity on spaces such as
$e^{-C/x^{2p}}L^2_g$ holds.
\end{thm}

\begin{rmk}\label{rmk:general-p}
While $\pi/2$ in the statement of the theorem might look peculiar, it
is purely geometric, and is explained in
Section~\ref{sec:conic-geometry}.

In addition, the function $x$ plays a dual role in the present discussion, as we
explain below: one is connected to the asymptotically conic geometry,
and is thus fixed, and other to the
analysis of the X-ray transform inversion; the latter determines the
exponential weight in the theorem. The arguments below will be
given in detail for $p=1$, in which case the decay assumption is
sufficiently fast Gaussian decay, i.e.\ $e^{-C/x^2}L^2_g$, with $C>0$
large. Working with general $p>0$ only requires minor changes, and we
will place these in remarks throughout the paper.
See Remark~\ref{rmk:p-1-cusp-ispsdo}.
\end{rmk}

Note that the assumption holds in particular on perturbations of
asymptotically Euclidean metrics (for which the link has conjugate
points at distance $\pi$), even though these typically have
conjugate points, indeed this is necessarily the case if the metric
keeps being asymptotic to Euclidean space but is not flat, as shown
recently by Guillarmou, Mazzucchelli and Tzou
\cite{Guillarmou-Mazzucchelli-Tzou:Conjugate}. This result thus partially
strengthens the injectivity result of Guillarmou, Lassas and Tzou
\cite{Guillarmou-Lassas-Tzou:Conic}, in that that work requires the
absence of conjugate points; however, this strengthening comes at the
cost of imposing faster decay conditions in our case.

If $\overline{M}$ has a global convex foliation, our Theorem combined
with the result of
\cite{Uhlmann-Vasy:X-ray} immediately implies the full invertibility
on sufficiently fast Gaussian decaying functions on $M$: first the
restriction to a collar neighborhood of the boundary is determined,
and thus if two functions have the same X-ray transforms, they are
supported away from $x=0$, so \cite{Uhlmann-Vasy:X-ray}  applies.

To explain the context of these results,
we note that it has been known for quite some time that under appropriate geometric
assumptions, namely the absence of conjugate points, $I^*I$ is an
elliptic pseudodifferential operator. For our purposes it is best to
consider $I$ as a map from (say, continuous) functions on $M$ to
functions on the sphere bundle $SM$, or equivalently (via the
Riemannian metric) the cosphere
bundle $S^*M$, as
$$
(If)(z,v) = \int f(\gamma_{z,v}(s))\, ds,
$$
where $\gamma_{z,v}$ is the geodesic through $z\in M$ with tangent
vector $v\in S_z M$. We then replace $I^*$ by the map $L$ from
functions on $SM$ to functions on $M$ defined by
$$
(Lw)(z)=\int_{S_z M} w(z,v) |d\sigma|,
$$
where $\sigma$ is a positive smooth density (e.g.\ the Riemannian one) on $S_zM$, smoothly dependent
on $z$; $LI$ is then an elliptic pseudodifferential operator of order $-1$.
This
gives that in the context of compact manifolds with boundary
satisfying these geometric conditions, $LI$, thus $I$, has a
finite, but potentially large, dimensional nullspace. The advance in the just mentioned papers
was to exclude the possibility of such a nullspace as well as to
localize the problem, thus eliminating the need for conjugate point assumptions. This was done by
introducing an artificial boundary, and recovering $f$ from $If$ from
information on geodesics that stay on one side of this boundary;
moving the artificial boundary sufficiently close to the original
boundary gave a small parameter in which asymptotic analysis
techniques could be used. Technically, this involved a localizer
$\tilde\chi$ on $SM$ which becomes singular at the artificial
boundary, localizing to geodesics that remain on the desired
side via the consideration of $L\tilde\chi I$. Based on the precise nature of the singularity, one gets a
different kind of an operator; with the particular choice made in
these papers, the approach relied on Melrose's
scattering algebra, associated to the new artificial boundary,
effectively pushing it to infinity analytically; we describe this
below.

In \cite{Vasy:Semiclassical-X-ray} a modified approach was introduced
where the artificial boundary was replaced by a semiclassical scaling
under somewhat more stringent geometric hypotheses; indeed, the two
approaches could even be combined, thus eliminating the extra
conditions and making sure that the only necessity for a combined approach is
purely geometric (as opposed to analytic). This new semiclassical approach is
more suited to our problem as otherwise it would be
harder to keep track of the behavior of the combined
pseudodifferential operator algebra when moving the artificial
boundary in this case as had been done in \cite{Uhlmann-Vasy:X-ray}: one has both an
artificial boundary, with the scattering algebra behavior, as well as a new
algebra at infinity, called the 1-cusp algebra. The
semiclassicalization of this joint algebra, on the other hand, easily
gives the full invertibility (rather than mere ellipticity) results
once the neighborhood of infinity is sufficiently small to control the
geometry, allowing for fixed artificial boundary; this is the key tool in the proof of Theorem~\ref{thm:main-conic}.

We prove new results of two different types. First, we develop a
new operator algebra, called the 1-cusp algebra, which is related to
the scattering pseudodifferential algebra but involves one more
blow-up, and its semiclassical version. Next, we show, similarly to
Uhlmann and Vasy \cite{Uhlmann-Vasy:X-ray}, that the X-ray
transform in the asymptotically conic setting can be modified to, via
composition with other operators,
an elliptic operator in this new algebra.  In the remainder of the
introduction in the two subsections we discuss each of these briefly.

\subsection{Analytic ingredients}
Before introducing the new 1-cusp algebra, we recall the scattering
algebra with which it shares many similarities. The scattering
pseudodifferential algebra was defined by Melrose in \cite{RBMSpec} in
the general geometric setting,
but his work had many predecessors. Indeed,  this algebra actually can be locally
reduced to a standard H\"ormander algebra, which in turn was studied earlier by
Parenti \cite{Parenti:Operatori} and Shubin
\cite{Shubin:Pseudodifferential}. Concretely then, on $\RR^n$, this
algebra arises by the standard quantization,
\begin{equation}\label{eq:sc-quant-non-compact}
(q_L(a)u)(z)=(2\pi)^{-n}\int_{\RR^n} e^{i(z-z')\cdot\zeta}a(z,\zeta)
u(z')\,d\zeta\,dz',\qquad u\in \calS(\RR^n),
\end{equation}
of symbols which are
separately symbolic, or symbolic of `product type', in the position and momentum variables
$(z,\zeta)$; symbols $a\in S^{m,l}$ of order $(m,l)$ satisfy
\begin{equation}\label{eq:sc-symbol-non-compact}
|(D_z^\alpha D_\zeta^\beta a)(z,\zeta)|\leq C_{\alpha\beta} \langle z\rangle^{l-|\alpha|}\langle\zeta\rangle^{m-|\beta|}
\end{equation}
for all multiindicies $\alpha,\beta\in\NN^n$. Thus, the Schwartz
kernel is the oscillatory integral (intepreted as a tempered distribution)
$$
K_A(z,z')=(2\pi)^{-n}\int_{\RR^n} e^{i(z-z')\cdot\zeta}a(z,\zeta)\,d\zeta
$$
with respect to the density $|dz'|$.
In Melrose's geometric
version one works on a compact manifold with boundary; the
correspondence is via the compactification $\overline{M}$ which we
have already described in the general context
of asymptotically conic spaces.

On a
manifold with boundary $M$ (we drop the bar over $M$ when we discuss
the general analytic structure on a manifold with boundary), scattering
vector fields $V\in\Vsc(M)$ are vector fields of the form $V=xV'$,
where $V'\in\Vb(M)$ is a vector field tangent to $\pa M$ (by the definition
of $\Vb(M)$), and where $x$ is a
boundary defining function. This notion is independent of the choice
of $x$ since any two choices differ by a positive factor.
In local coordinates near $\pa M$, with $y$ coordinates on $\pa M$,
scattering vector fields are thus of the
form
$$
a_0(x,y)x^2D_x+\sum_{j=1}^{n-1} a_j(x,y)xD_{y_j}.
$$
Scattering differential operators are finite sums of products of such
vector fields.
Scattering pseudodifferential operators are a generalization of the
latter, formally replacing polynomials in the vector fields by more
general functions. More precisely, their Schwartz kernels $K$ are locally given by oscillatory integrals
of the form
$$
(2\pi)^{-n}\int e^{i\Big(\frac{x-x'}{x^2}\xisc+\frac{y-y'}{x}\cdot\etasc\Big)}a(x,y,\xisc,\etasc)\,d\xisc\,d\etasc
$$
with respect to the density $\frac{|dx'\,dy'|}{(x')^{n+1}}$, see
\cite{RBMSpec}. Here the symbolic estimates of \eqref{eq:sc-symbol-non-compact} become conormal estimates
for $a$
$$
|((x\pa_x)^j\pa_y^\alpha \pa_{\xisc}^k\pa_{\etasc}^\beta a)(x,y,\xisc,\etasc)|\leq
C_{jk\alpha\beta}\langle\xisc,\etasc\rangle^{m-k-|\beta|}x^{-l}.
$$
The Schwartz kernel $K$ can in turn be regarded as a well-behaved,
namely conormal, distribution on a resolution (double blow-up) of the double space
$M^2=M\times M$. That is, in variables
$x,y,X=\frac{x-x'}{x^2},Y=\frac{y-y'}{x}$, $K$ is conormal to the
(lifted) diagonal $\{X=0,\ Y=0\}$. We discuss this resolution in some
detail in Section~\ref{sec:scattering}. One advantage of the geometric
approach, making the definition via a resolution, is that it is
automatically invariantly defined, i.e.\ from the local perspective it is
diffeomorphism invariant.

One reason that the scattering
algebra is a useful object to work with is that it is not only a
bi-filtered $*$-algebra, thus closed under composition, but the
composition, to leading order, modulo $\Psisc^{m-1,l-1}$, is symbolic, i.e.\ it can be 
expressed algebraically in terms
of the principal symbols, meaning the class $[a]$ of $a$ in \eqref{eq:sc-quant-non-compact} modulo $S^{m-1,l-1}$. Furthermore, the mapping properties are also very
well behaved. Defining weighted Sobolev spaces $H^{s,r}$ by adding a  weight,
$H^{s,r}= \langle z \rangle^{-r} H^s$, where in the case of $\RR^n$
$H^s$ is the standard Sobolev space (and in general is transported to
the asymptotically conic space via local coordinates), any scattering pseudodifferential operator $A\in
\Psisc^{m,\ell}$, maps weighted Sobolev spaces to weighted Sobolev
spaces $A: H^{s,r} \to H^{s - m, r - \ell}$. The residual class
$\Psisc^{-\infty, -\infty}$ maps any $H^{s,r} \to H^{s',r'}$, so
that they are all compact operators on any $H^{s,r}$.

The 1-cusp algebra shares many of the useful properties.
In order to define the 1-cusp algebra, it is helpful to first consider
the corresponding vector fields. Recall that cusp
vector fields are defined on a manifold with boundary equipped with a
boundary defining function $x$ with a given differential at the boundary
(so the boundary defining function is determined up to $O(x^2))$;
$V\in\Vcu(M)$ are smooth vector fields tangent to $\pa M$ with the
property that $Vx=O(x^2)$.
In local coordinates thus they are of the
form
$$
a_0(x,y)x^2D_x+\sum a_j(x,y)D_{y_j}.
$$
We then define the 1-cusp vector fields as $\Vocu(M)=x\Vcu(M)$,
so in local coordinates thus they are of the
form
$$
a_0(x,y)x^3D_x+\sum a_j(x,y)xD_{y_j}.
$$
These are thus also scattering vector fields, but with an additional
order of vanishing in the $D_x$ component.

When turning to the 1-cusp pseudodifferential operators, again defined
on manifolds with a preferred boundary defining function $x$ fixed up
to $O(x^2)$, corresponding to this
additional vanishing, we blow up the Schwartz kernel double space from
the scattering coordinates
$(x,y,X, Y)$ to $(x,y, V = \frac{X}{x}, Y )$. 
As $(x^2 \pa_x, x \pa_y)$ and $(X, Y)$ corresponded to $(\xisc, \etasc)$
in the scattering case, $(x^3 \pa_x, x \pa_y)$  and $(V, Y)$
correspond to $(\xiocu, \etaocu)$ in this new class when we write an
oscillatory integral to define an operator using the symbol. We define
our new class of symbols as smooth functions $a^{\oc}(x,y,\xiocu, \etaocu)$
which satisfy the inequalities
\[
  \left| (x \pa_x)^{j}  \pa_y^{\alpha}\pa^k_{\xiocu} \pa^\beta_{\etaocu}
    a^{\oc}(x,y,\xiocu,\etaocu)\right| \le C_{ab} \langle \xiocu, \etaocu \rangle^{m
    - k - |\beta|} x^{-\ell} .
\]
While in coordinates this is the same definition as a scattering class symbol (conormal
symbols), invariantly these are symbols (of the same type) on a
different (scaled)
cotangent bundle, and correspondingly we use a
different quantization map to turn them into operators. Concretely, we
use the
resolved space, assuring that they specify conormal distributions with
respect to $\{V=0,\ Y=0\}$:
\begin{equation*}\begin{aligned}
  K_A(x,y,V, Y) &= \frac{1}{(2\pi)^n}\int a(x,y,\xiocu, \etaocu)
  e^{i V \xiocu + Y \cdot \etaocu} \, d\xiocu \, d\etaocu\\
  &= \frac{1}{(2\pi)^n}\int a(x,y,\xiocu, \etaocu)
  e^{i \Big(\frac{x-x'}{x^3}\xiocu + \frac{y-y'}{x}\cdot \etaocu\Big)} \, d\xiocu \, d\etaocu,
  \end{aligned}\end{equation*}
with respect to the density $\frac{|dx'\,dy'|}{(x')^{n+2}}$.

Then, we show that we can describe composition of operators in this new algebra symbolically, and that we can define ellipticity similarly and construct parametrices for elliptic operators in this class with residual errors.  As in the scattering algebra, these residual errors are compact operators.

\begin{prp}[\cite{Zachos:Thesis}, see Proposition~\ref{prop:1c-comp}]
If $A \in \Psiocu^{m, \ell}$ with principal symbol, modulo
$\Psiocu^{m-1,\ell-1}$, $[a]\in S^{m,\ell}/S^{m-1,\ell-1}$ and $B\in \Psiocu^{m',\ell'}$ with principal symbol $[b]$ then $A \circ B \in \Psiocu^{m + m', \ell + \ell'}$  with principal symbol $[a][b]=[ab]$. 
\end{prp}

As usual, this implies that there is a parametrix for elliptic
operators:

\begin{prp}[\cite{Zachos:Thesis}, see Proposition~\ref{prop:1c-param}]
If $A\in \Psiocu^{m,\ell}$ with principal symbol $a$ is
elliptic, i.e.\ for some $c>0$,
\[ |a(x,y,\xiocu,\etaocu)| \ge C \langle \xiocu,\etaocu\rangle ^m x ^{-\ell} \text{ for } |(\xiocu,\etaocu)|\gg 1 \text{ or } x \ll 1\]
then there is a parametrix $B\in \Psiocu^{-m,-\ell}$ with error $AB-I,BA-I$ in $\Psiocu^{-\infty, -\infty}$.
\end{prp}

As already alluded to earlier in the introduction, a semiclassical
variant of the 1-cusp algebra plays a key role in this work. This is a
foliation semiclassical algebra, associated to the full foliation $\cF$ by the
level sets of $x$. The foliation semiclassical algebra was described in
\cite{Vasy:Semiclassical-X-ray} in both the standard (no boundary) and
scattering (artificial boundary) settings; here we thus focus on the
1-cusp aspects.
Near $x=0$, the foliation tangent 1c-vector
fields are locally
$$
\sum a_j(x,y)xD_{y_j};
$$
the collection of these is denoted by $\Vocu(M;\cF)$.
The semiclassical version of $\Vocu(M)$ is
simply $\Vocuh(M)=h\Vocu(M)$; the semiclassical foliation version is
$$
\Vocuhh(M;\cF)=h\Vocu(M)+h^{1/2}\Vocu(M;\cF).
$$
Thus, the semiclassical foliation 1-cusp differential operators
take the form
$$
\sum_{\alpha+|\beta|\leq m} a_{\alpha\beta}(x,y,h)(hx^3D_x)^{\alpha} (h^{1/2}xD_y)^{\beta}.
$$
The corresponding pseudodifferential operators $A\in\Psiocuhh^{m,l}(M,\cF)$
again arise by a modified semiclassical quantization of standard
semiclassical symbols $a$, i.e.\ ones satisfying (conormal in $x$)
symbol estimates
$$
|(xD_x)^\alpha D_y^\beta D_{\xiocu}^\gamma D_{\etaocu}^\delta
a(x,y,\xiocu,\etaocu,h)|\leq C_{\alpha\beta\gamma\delta} \langle
(\xiocu,\etaocu)\rangle^{m-\gamma-|\delta|} x^{-l},
$$
namely
\begin{equation}\begin{aligned}\label{eq:1c-quantize}
    A_h u(x,y)&=Au(x,y,h)\\
    &=(2\pi)^{-n} h^{-n/2-1/2}\int
    e^{i\Big(\frac{x-x'}{x^3}\frac{\xiocu}{h}+\frac{y-y'}{x}\frac{\etaocu}{h^{1/2}}\Big)}\\
    &\qquad\qquad\qquad\qquad\qquad a(x,y,\xiocu,\etaocu)\,u(x',y')\,\frac{dx'\,dy'}{(x')^{n+2}}\,d\xiocu\,d\etaocu.
  \end{aligned}\end{equation}
Thus, in $x>0$, these are just the standard semiclassical foliation
operators, in $h>0$ the standard 1-cusp pseudodifferential
operators, with the combined behavior near $x=h=0$. In particular we
have an elliptic theory as in the semiclassical foliation setting: if
$A$ is elliptic, meaning
$$
|a(x,y,\xiocu,\etaocu)|\geq cx^{-l}\langle(\xiocu,\etaocu)\rangle^m,\qquad c>0,
$$
then there is a parametrix $B\in\Psiocuhh^{-m,-l}(M,\cF)$ with
$$
AB-\Id,BA-\Id\in h^\infty\Psiocuhh^{-\infty,-\infty}(M,\cF),
$$
and there exists
$h_0>0$ such that for $h<h_0$, $A\in\cL(\Hocuh^{s,r},\Hocuh^{s-m,r-l})$
is invertible with uniform bounds. This is the real reason for the
usefulness of the semiclassical setting: the errors of a parametrix are not only
compact (or finite rank), but can be eliminated altogether.

\subsection{Conic geometry and inverse problems}\label{sec:conic-geometry}

Bicharacteristics are integral curves of the Hamilton vector field of
the dual metric function of $g$. For our asymptotically conic metrics
the dual metric function is naturally a function on the same bundle
$\Tsc^*M$, on which principal symbols of the scattering
pseudodifferential operators live. When discussing the geometry,
however, we will use $(\tau,\mu)$ rather than $(\xisc,\etasc)$ as
coordinates on the fibers of this bundle, i.e.\ we write covectors as
$$
\tau\frac{dx}{x^2}+\mu\cdot\frac{dy}{x}.
$$
This separate notation, in particular, serves to emphasize that these
geometric objects will be unchanged even if one uses a different
analytic scaling (when for analytic purposes $x$ is replaced by $x^p$), cf.\ Remark~\ref{rmk:general-p}.
A computation of Melrose
\cite{RBMSpec} gives that for asymptotically conic metrics $g$,
$$
\frac{1}{2}H_g=x\Big(\tau(x\pa_x+\mu\cdot\pa_\mu)-|\mu|^2\pa_\tau+\frac{1}{2}H_h+xV\Big),
  $$
  where $V$ is a vector field tangent to the boundary $x=0$; this
  gives the arclength parametrization of geodesics. In view
  of the overall $x$ factor, which makes this parameterization
  degenerate at the boundary, it is useful to work with
  $\scH_g=x^{-1}H_g$. While
  $$
\frac{1}{2}\scH_g=\tau(x\pa_x+\mu\cdot\pa_\mu)-|\mu|^2\pa_\tau+\frac{1}{2}H_h+xV
  $$
  is non-vanishing at $x=0$ in general, it still does vanish at
  $\{x=0,\ \mu=0\}$, which is called the radial set. When $g$ is conic, and thus
$V=0$, within the unit level set of
the dual metric function of $g$, $\mu=0$ means $\tau=\pm 1$, i.e.\ at
such a point $\scH_g$ is a (non-vanishing) multiple of the radial
vector field $x\pa_x$, and thus the geodesic is radial (the $Y$
component, $y$, is
constant along it). As long as one stays microlocally away from
these radial geodesics, as we do here, one can work with
$\frac{1}{2|\mu|}\scH_g$ in place of $\scH_g$; we explain this in
Section~\ref{sec:normal-op} in terms of a blow-up.
This amounts to a reparameterization of the
  integral curves $c=c(s)$ of $\frac{1}{2}H_g$ via $\frac{dr}{ds}=x(c(s))|\mu(c(s))|$, i.e.\ if the
  reparameterized bicharacteristics are $\gamma=\gamma(r)$, then $\frac{ds}{dr}=x(\gamma(r))^{-1}|\mu(\gamma(r))|^{-1}$.
Melrose and Zworski \cite{RBMZw} computed these reparameterized
bicharacteristics. Note that switching between $\frac{1}{2}\scH_g$ and
  $\frac{1}{2|\mu|}\scH_g$ is a smooth reparameterization away from
  $\mu=0$, so one can equally well use either of these in that region.
  To
leading order at $x=0$, so globally for actually conic metrics, the
interior, originally
unit speed (prior to reparameterization),
bicharacteristics can be written
as follows:
\begin{equation}\begin{aligned}\label{eq:conic-bichar}
    &    x=\frac{x_0}{\sin r_0}\sin (r+r_0),\ \tau=\cos (r+r_0),\ |\mu|=\sin (r+r_0),\\
    & (y,\hat\mu)=\exp(rH_{\frac{1}{2} h})(y_0,\hat\mu_0),\ r\in(-r_0,-r_0+\pi),
  \end{aligned}\end{equation}
with $(y,\hat\mu)$ thus following a unit speed lifted geodesic of
length $\pi$ in
$Y$. Note that the maximum of $x\circ \gamma$, which is the point of tangency to
level sets of the function $x$, occurs halfway in the domain of
$\gamma$, at (parameter) distance $\pi/2$ from either endpoint, at $r+r_0=\pi/2$, and
thus in terms of the boundary geodesic distance $\pi/2$ from either
endpoint. In particular, near this point, where most of the action
takes place for us, $r$ and $t$ (the parameterization for
$\frac{1}{2}\scH_g$) can be used equally well. This also explains the
$\pi/2$ in the statement of our main Theorem~\ref{thm:main-conic}.

With these geometric preliminaries and with the properties of the new algebra established, we turn to the
X-ray transform $I$ and the operator $L$ defined earlier as a
replacement for $I^*$.
 
It turns out that for asymptotically conic metrics, if one uses a
suitable localizer $\tilde\chi$, and conjugates $L\tilde\chi
I$ by suitable exponential weights $e^{\Phi}$ to define the modified
normal operator, the result is an element of our new algebra.
Here $\tilde\chi$ localizes to (has support near) points in the sphere
bundle which are almost tangent to level sets of the boundary function
$x$. More precisely, the angle to the level sets goes to $0$ as $x\to 0$ proportionally
to $x$, i.e.
$$
\tilde\chi=\tilde\chi(x,y,\lambda/x,\omega)
$$
with compact support in the third slot, where we write tangent vectors as $\lambda
(x\pa_x)+\omega\cdot\pa_y$ relative to a product decomposition near
$\pa M$ respecting the foliation, see
Section~\ref{sec:geodesic-structure} for detail.
The exponential weight $e^\Phi$, on the other hand, is
Gaussian decaying, concretely $\Phi=-\frac{1}{2x^2}$. As
$$
L\tilde\chi
If=e^{\Phi}(e^{-\Phi}L\tilde\chi
Ie^{\Phi})e^{-\Phi}f,
$$
this means that the results we obtain are for $e^{-\Phi}f$, with its
Gaussian growing weight, which means that on the one hand the
estimates are strong at infinity, but on the other hand they only
apply to Gaussian decaying functions $f$. The actual analytic result, with an
ellipticity statement, is:

\begin{thm}[cf.\ \cite{Zachos:Thesis}, and see Theorem~\ref{ispsdo}
  for the full semiclassical version]
For asymptotically conical metrics with cross sections without
conjugate points within distance $\pi/2$ and for suitable localizers $\tilde\chi$, the modified normal operator of the X-ray transform  is an elliptic operator (for sufficiently small $x$) in the 1-cusp algebra. 
\end{thm}

\begin{rmk}\label{rmk:xp-1-cusp}
If we replace $x$ by $x^p$ in the definition of the rescaled
$\lambda$, i.e.\ take $\tilde\chi=\tilde\chi(x,y,\lambda/x^p,\omega)$,
and similarly $\Phi=-\frac{1}{2px^{2p}}$, the conclusions remain valid,
with the $x^p$-based 1-cusp algebra, i.e.\ the one defined in an
analogous manner but with the boundary defining function $x$ replaced
by $x^p$, with a corresponding change of the smooth structure. Since we only need conormal
behavior (as opposed to smoothness) of the coefficients of the algebra
at $x$, the dependence of $\tilde\chi$ on $x$ vs.\ $x^p$ is
immaterial.
See Remark~\ref{rmk:1cusp-phase-exp-p} for the key analytic reason.
\end{rmk}

As just described for the 1-cusp pseudodifferential algebra, this implies that there is a parametrix with a finite rank
error acting on functions supported sufficiently close to infinity and
with sufficiently fast decay (which arises from the modifications
discussed below). In order to remove this error, we proceed by fixing
an artificial boundary $x=c$, $c>0$ sufficiently small, fixed by
geometric considerations, namely the lack (in a precise sense
discussed in Section~\ref{sec:X-ray-inv}) of conjugate points in $x<c$. Now we are
on a manifold with boundary, with the two boundary hypersurfaces given
by $x=0$ and $x=c$, so can in particular consider the
pseudodifferential algebra which is 1-cusp at $x=0$ and scattering at
$x=c$, corresponding to the artificial boundary there. The approach of
\cite{Uhlmann-Vasy:X-ray} would be to allow $c$ to become even smaller
and use it as an asymptotic parameter. Instead, as already stated, we
regard $c$ as fixed, but introduce a semiclassical parameter $h$ in
the spirit of \cite{Vasy:Semiclassical-X-ray}. In fact, we need
some additional information, namely we use the full foliation $\cF$ by the
level sets of $x$ in $0\leq x\leq c$; this allows one to define the
semiclassical foliation version of the 1-cusp/scattering algebra.

The main technical result, Theorem~\ref{ispsdo}, is that $A$, given by an exponential
conjugate of $L\tilde\chi I$, is elliptic in the
1-cusp algebra, and indeed in the semiclassical foliation 1-cusp
algebra. Here we use
$$
\tilde\chi=\tilde\chi(x,y,\lambda/(h^{1/2}x),\omega)
$$
and $e^{\Phi}$ with
$\Phi=-\frac{1}{2hx^2}$.

\begin{thm}[See Theorem~\ref{ispsdo}]
For asymptotically conical metrics with cross sections without
conjugate points within distance $\pi/2$ and for suitable localizers $\tilde\chi$, the modified normal operator
of the X-ray transform  is an elliptic operator (for sufficiently
small $h$ and $x$) with respect to the semiclassical foliation 1-cusp algebra. 
\end{thm}

\begin{rmk}\label{rmk:p-1-cusp-ispsdo}
The semiclassical foliation version of Remark~\ref{rmk:xp-1-cusp} is
applicable; in this case
$\tilde\chi=\tilde\chi(x,y,\lambda/(h^{1/2}x^p),\omega)$ and
$\Phi=-\frac{1}{2phx^{2p}}$.
See Remark~\ref{rmk:1cusp-phase-exp-p} for the key analytic reason.
  \end{rmk}

Thus, we can construct a parametrix, whose error is actually
small for small $h$, implying invertibility. This immediately gives

\begin{thm}[See Corollaries~\ref{cor:main-conic-op-support} and \ref{cor:main-conic-support}]
  For manifolds as specified above, the original geodesic
  X-ray normal operator and thus the X-ray transform itself, acting on
  functions with Gaussian decay, will have a trivial
  nullspace supported in  $x< \bar x$. \end{thm}

The final ingredient of the proof of our main theorem,
Theorem~\ref{thm:main-conic}, is to eliminate the support condition
by working with a combined scattering-1 cusp algebra; see Section~\ref{sec:combined-sc-1c}.

\section{The 1-cusp algebra and its semiclassical version}
We proceed to create a new pseudodifferential algebra, the 1-cusp algebra, by performing
blow-ups on the Schwartz kernel double space, and also discuss it in terms of
explicit quantizations. This pseudodifferential operator algebra is
local on the underlying manifold with boundary $M$, unlike say Melrose's b-algebra, or more
relevantly, the cusp algebra, and can thus
be described by explicit quantization and diffeomorphism invariance
considerations, much as the case of the scattering algebra. We also
discuss, as in the case of the scattering and cusp algebras, the connection to a
class of pseudodifferential operators on $\RR^n$. Yet an
alternative approach to a description of certain (limited) aspects of
this algebra would be to follow the work of Amman, Lauter, Nistor \cite{Amman-Lauter-Nistor:Pseudodifferential}, which uses Lie algebroids.

\subsection{The scattering double space and the scattering pseudodifferential algebra}\label{sec:scattering}
As a convenience for the reader, we restate some basic definitions and
properties of the scattering algebra, which serves as a potential
starting point for our new algebra, and shares some properties with it. For further details, refer to Melrose's original paper introducing the scattering algebra \cite{RBMSpec}.

 Melrose defined the scattering algebra on general manifolds with
 boundary; a motivation is that the Laplacian of an asymptotically
 conic (in particular an asymptotically Euclidean) Riemannian metric
 is an element of this algebra. Let $x$ be a boundary defining
 function of $M$; this is determined up to a smooth positive
 factor. Then $\Vsc(M)$ consists of $\CI$ vector fields of the form
 $xV'$, $V'\in\Vb(M)$, i.e.\ $V'$ is a smooth vector field tangent to
 $\pa M$. In local coordinates $(x,y)$ near $\pa M$, with $y$ local
 coordinates on $\pa M$, scattering vector fields $\Vsc$
 are those vector fields generated, over $\CI(M)$, by $\{x^2 \pa_x, x
 \pa_{y_1},\ldots,x\pa_{y_{n-1}} \}$, i.e.\ are of the form
 $$
a_0(x,y)(x^2 \pa_x)+\sum_{j=1}^{n-1}a_j(x,y)(x\pa_{y_j}).
$$
These are thus all smooth sections of  a vector bundle, the scattering
tangent bundle, $\Tsc M$, whose elements at any point $p\in M$ can be written as $\lambda
(x^2\pa_x)+\sum_{j=1}^{n-1}\omega_j (x\pa_{y_j})$, i.e.\ $(\lambda,\omega)$ are coordinates
on the fibers of this vector bundle. The scattering cotangent bundle
is then the dual vector bundle, and we can thus write
scattering covectors as
$$
\tau\,\frac{dx}{x^2}+\mu\cdot\frac{dy}{x},
$$
i.e.\ $(\tau,\mu)$ are local coordinates on the fibers of $\Tsc^*M$, and
$(x,y,\tau,\mu)$ on $\Tsc^*M$ itself. While we keep this notation for
the asymptotically conic geometric discussion, in the analytic context
we will use the notation $(\xisc,\etasc)=(\tau,\mu)$, i.e.\ covectors
are written as
$$
\xisc\,\frac{dx}{x^2}+\etasc\cdot\frac{dy}{x}.
$$

The Schwartz kernel of a scattering pseudodifferential operator is a conormal distribution
scattering double space $M_{\scl}^2$, which is a blow-up, or
resolution, of the standard double space $M^2=M\times M$; see Figure~\ref{fig:sc-double}. We recall
that the blow-up of a product-type, or p-, submanifold of a manifold
with corners is a new manifold with corners in which different {\em
  normal} directions of approach to the submanifold being blown up are
distinguished; this process is thus an invariant generalized version
of the introduction of spherical coordinates around a submanifold,
i.e.\ of cylindrical coordinates. (Melrose's paper \cite{RBMSpec}
contains details of the blow-up process; see also
\cite{Melrose:Atiyah}, or indeed \cite[Section~5]{Vasy:Zero-energy} in a context that will play a
role in Section~\ref{sec:algebra}.) The double space is constructed by
taking $M^2=M\times M$ and then first
blowing up the corner $(\pa M)^2$ to get the b
double space, a manifold with corners, where $(\pa M)^2$ has been
blown up into the \emph{b front face}. In the interior of the b front
face, near the diagonal, we have coordinates
$$
x,y, \frac{x-x'}{x},y',
$$
with $y$ local coordinates on $\pa M$. The lifted diagonal
$\{\frac{x-x'}{x}=0,\ y=y'\}$ only meets the interior of this b front face,
i.e.\ $\frac{x-x'}{x}$ is bounded away both from $1$ and $-\infty$
along it. (Near the lift of $\pa M\times M$, i.e.\ $\{x=0\}$, and
$M\times\pa M$, i.e.\ $\{x'=0\}$, we need to use somewhat different
coordinates, but being near these faces amounts to $\frac{x'}{x}$
tending to $+\infty$, resp.\ $0$.)
Then, a second blow up, of the
boundary $\{x=0,\ \frac{x-x'}{x}=0,\ y=y'\}$, of the lifted diagonal
is performed with the new front face being the \emph{scattering
  front-face}. We obtain coordinates
$$
x,y,X=\frac{x-x'}{x^2}, Y=\frac{y-y'}{x},
$$
near the interior. (One can also replace $x,y$ by $x',y'$, and below
we discuss another possibility.)

\begin{figure}[ht]\begin{center}\label{hmmm}  \includegraphics[width=80mm]
{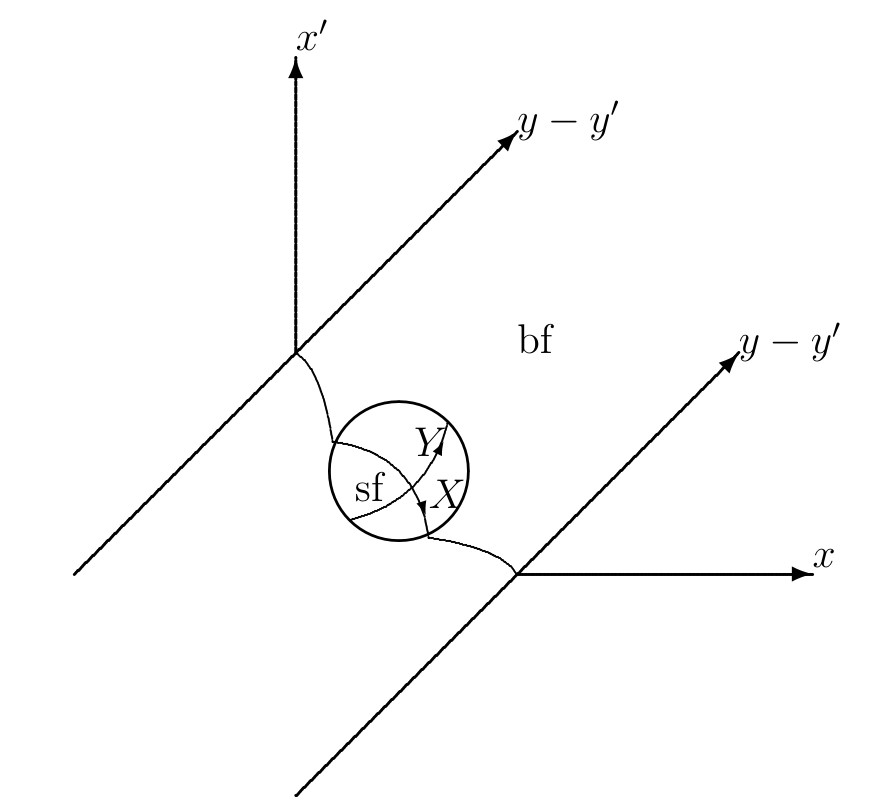}\end{center}\caption{The scattering double space,
with $\mathrm{\bf}$ the b-front face, $\mathrm{sf}$ the scattering
front face.}\label{fig:sc-double}\end{figure}

The scattering algebra consists of operators whose Schwartz kernels
 on this blown-up double space are well-behaved, meaning they are
 $\CI$ in the interior of $M^2$ away from the diagonal, are conormal
 to all boundary faces with infinite order vanishing at all of them
 except the scattering front face, and have a
 conormal singularity along the diagonal. Explicitly, the latter means
 that one can write the Schwartz kernels near the lifted diagonal
 intersecting the scattering front face, relative to the density
 $\frac{|dx'\,dy'|}{(x')^{n+1}}$, as
\begin{equation}\label{eq:sc-SK-XY}
K_A(x,y,X,Y)=(2\pi)^{-n}\int e^{i(\xisc X+\etasc\cdot Y)}a(x,y,\xisc,\etasc)\,d\xisc\,d\etasc,
\end{equation}
where $a\in S^{m,l}$ is a `product type' symbol
\begin{equation*}
  \left|(x \pa_x)^{j}  \pa_y^{\alpha} \pa^k_{\xisc}
    \pa^\beta_{\etasc} a(x,y,\xisc,\etasc)\right| \le C_{jk\alpha \beta} \langle
  \xisc, \etasc \rangle^{m - k - |\beta|} x^{-\ell}.
\end{equation*}
 We emphasize that this description is a priori only valid in a
 neighborhood of the lifted diagonal, but in fact is also valid in a
 neighborhood of the scattering front face, though not globally. We
 give below a version that is global in charts $O\times O$, $O$ open
 in $M$, via a reduction to $\overline{\RR^n}$.
 
One very convenient feature of the scattering pseudodifferential
algebra is that it can in fact be locally reduced to a standard
H\"ormander algebra \cite{Hormander:v3} on $\RR^n$, where `locally' is understood on the
radial compactification $\overline{\RR^n}$, resp.\ the compact
manifold with boundary $M$. Namely, taking `product type' symbols $a\in S^{m,l}$ on $\RR^n\times\RR^n$ such that
$$
|\pa_z^\alpha\pa_\zeta^\beta a(z,\zeta)|\leq C_{\alpha\beta} \langle
z\rangle^{l-|\alpha|}\langle \zeta\rangle^{m-|\beta|},
$$
and defining the Schwartz kernel of the standard, say, left quantization,
\begin{equation}\label{eq:sc-SK-Rn}
K_A(z,z')=(2\pi)^{-n}\int e^{i(z-z')\cdot\zeta}a(z,\zeta)\,d\zeta
\end{equation}
relative to the density $|dz'|$, the scattering algebra is obtained,
modulo operators with a Schwartz Schwartz kernel on $M^2$, by
identifying neighborhoods of points on $M$ with corresponding
neighborhoods on $\overline{\RR^n}$, and pulling back the Schwartz
kernel of an operator given by the just described left
quantization. The principal symbol of $A$ is defined as the equivalence class
$[a]$ of $a$ in $S^{m,l}/S^{m-1,l-1}$.

In order to connect the $\RR^n$-based description to the geometric
one, it is useful to use yet different coordinates near the scattering
front face, namely
$$
x,\ y,\ \tilde X=\frac{1}{x}-\frac{1}{x'}=\frac{x'-x}{xx'},\ \tilde Y=\frac{y}{x}-\frac{y'}{x'}=\frac{y-y'}{x}+\Big(\frac{1}{x}-\frac{1}{x'}\Big)y',
$$
so
\begin{equation*}\begin{aligned}
    &\tilde X=-\frac{x}{x'}X=-(1-xX)^{-1} X,\\
    &\tilde Y=Y -(1-xX)^{-1} X y'=Y -(1-xX)^{-1} X (y-xY),
\end{aligned}\end{equation*}
showing the smoothness of $\tilde X,\tilde Y$, and the reverse
expressions are also similarly checked to be smooth. In particular,
notice that $\tilde X=-X$, $\tilde Y=Y-Xy$ at $x=0$. Hence
\eqref{eq:sc-SK-XY} can be equally well-described as
\begin{equation}\label{eq:sc-SK-tildes}
K_A(x,y,X,Y)=(2\pi)^{-n}\int e^{i(\widetilde{\xisc} \tilde X+\widetilde{\etasc} \cdot
  \tilde Y)}\tilde a(x,y,\widetilde{\xisc},\widetilde{\etasc})\,d\widetilde{\xisc}\,d\widetilde{\etasc},
\end{equation}
with $\tilde a\in S^{m,l}$ as well. To reduce this to the
$\overline{\RR^n}$ perspective, recall that where, say, $|z_n|$ is
relatively large (and say $z_n$ is positive), we can use
$z_n^{-1},\frac{z_j}{z_n}$ as coordinates on the radial
compactification, with $z_n^{-1}$ defining the boundary; the
correspondence then is letting $x=z_n^{-1}$, $y_j=\frac{z_j}{z_n}$,
$1\leq j\leq n-1$, so
$z_n=x^{-1}$, $z_j=y_j/x$, so \eqref{eq:sc-SK-tildes} is in fact the
same as \eqref{eq:sc-SK-Rn}, keeping in mind that
$|dz'|=\frac{|dx'\,dy'|}{(x')^{n+1}}$. An advantage thus of
\eqref{eq:sc-SK-tildes} as well as \eqref{eq:sc-SK-Rn} is that they
are not restricted to the interior of the b-front face; they are also
valid at the left and right faces, at least near the diagonal in
$M^2$, i.e.\ on sets of the form $O\times O$, $O$ a coordinate chart,
thus eliminating dividing up treatments of the Schwartz kernels into
several regions, such the interior of the scattering front face, the
boundary of the scattering front face, etc.

One of the most significant features about the scattering algebra in
contrast to its many relatives (such as the b-algebra \cite{MR3792086}
or the 0-algebra \cite{MR1133743}) is that composition can be
described algebraically in terms of symbols to leading order in every
sense; this is immediate from the $\RR^n$-based description above.

\begin{prp}
If $A \in \Psisc^{m, \ell}$ with principal symbol, modulo
$\Psisc^{m-1,\ell-1}$, $[a]\in S^{m,\ell}/S^{m-1,\ell-1}$ and $B\in \Psisc^{m',\ell'}$ with principal symbol $[b]$ then $A \circ B \in \Psisc^{m + m', \ell + \ell'}$  with principal symbol $[a][b]=[ab]$. 
\end{prp}

\begin{prp}
If $A\in \Psisc^{m,\ell}$ with principal symbol $a$ is
elliptic, i.e.\ for some $c>0$,
\[ |a(x,y,\xisc, \etasc)| \ge C \langle \xisc, \etasc\rangle ^m x ^{-\ell} \text{ for } |(\xisc, \etasc)|\gg 1 \text{ or } x \ll 1\]
then there is a parametrix $B\in \Psisc^{-m,-\ell}$ with error $AB-I,BA-I\in\Psisc^{-\infty, -\infty}$.
\end{prp}

We define the usual weighted Sobolev spaces $H^{s,r}$ on $\RR^n$ by
imposing a  weight: $H^{s,r}= \langle z \rangle^{-r} H^s$; these are
then transported to the manifold via the just discussed identification
to define the scattering Sobolev spaces $\Hsc^{s,r}$. Equivalently, for
any real $s,r$, writing $\cS'$ for the space of tempered distributions
on $M$, i.e.\ the dual of $\CI$ functions vanishing to infinite order
at $\pa M$,
$$
\Hsc^{s,r}=\{u\in\cS':\ \exists A\in\Psisc^{s,r}\ \text{elliptic}\
\text{and}\ Au\in L^2\},
$$
where $L^2$ is with respect to a scattering density
$\frac{|dx\,dy|}{x^{n+1}}$, which corresponds to $|dz|$ in the
identification on $\RR^n$.

These weighted Sobolev spaces can be used to describe the mapping
properties of scattering pseudodifferential operators:  

\begin{prp}
  If $A\in \Psisc^{m,\ell}$, then $A: \Hsc^{s,r} \to
  \Hsc^{s - m, r - \ell}$.
\end{prp}

Because the parametrix error is not only smoothing (order $-\infty$ in
the differential sense) but includes a restriction on growth rates,
the error is actually compact on any weighted Sobolev space and we can get desired Fredholm properties.

For example, $\sigma(\Delta+1) = \xisc^2 + |\etasc|^2 +1$ is elliptic in
the scattering algebra but $\sigma(\Delta - 1)  = \xisc^2 + |\etasc|^2 -
1$ is not. Both of these operators are elliptic in the standard sense,
but the scattering algebra explains why one operator has an infinite
dimensional tempered distributional nullspace and the other does not.

\subsection{The cusp double space and algebra}
We now recall the definition of the cusp pseudodifferential algebra
and its properties. It is defined on manifolds with boundary $M$ with
a boundary function $x$ defined up to adding an element of
$x^2\CI(M)$, i.e.\ any other alternative boundary defining function
for this structure is
of the form $\tilde x=x+x^2\phi$, with $\phi\in\CI(M)$. In order to do
so, we start by discussing the double space, as appears in the work of
Mazzeo and Melrose \cite{Mazzeo-Melrose:Fibred}. Indeed, these authors
provide a joint framework for the scattering and the cusp algebras
within the class of `fibred cusp' algebras. As a reference to the
terminology of this paper we mention that in the scattering algebra
case the corresponding boundary fibration is the identity map, while
in the cusp algebra setting the boundary fibration is the map that
sends every point on the boundary to a single point (so the fiber is
the whole boundary).

\begin{figure}[ht]
\begin{center}
\includegraphics[width=90mm]{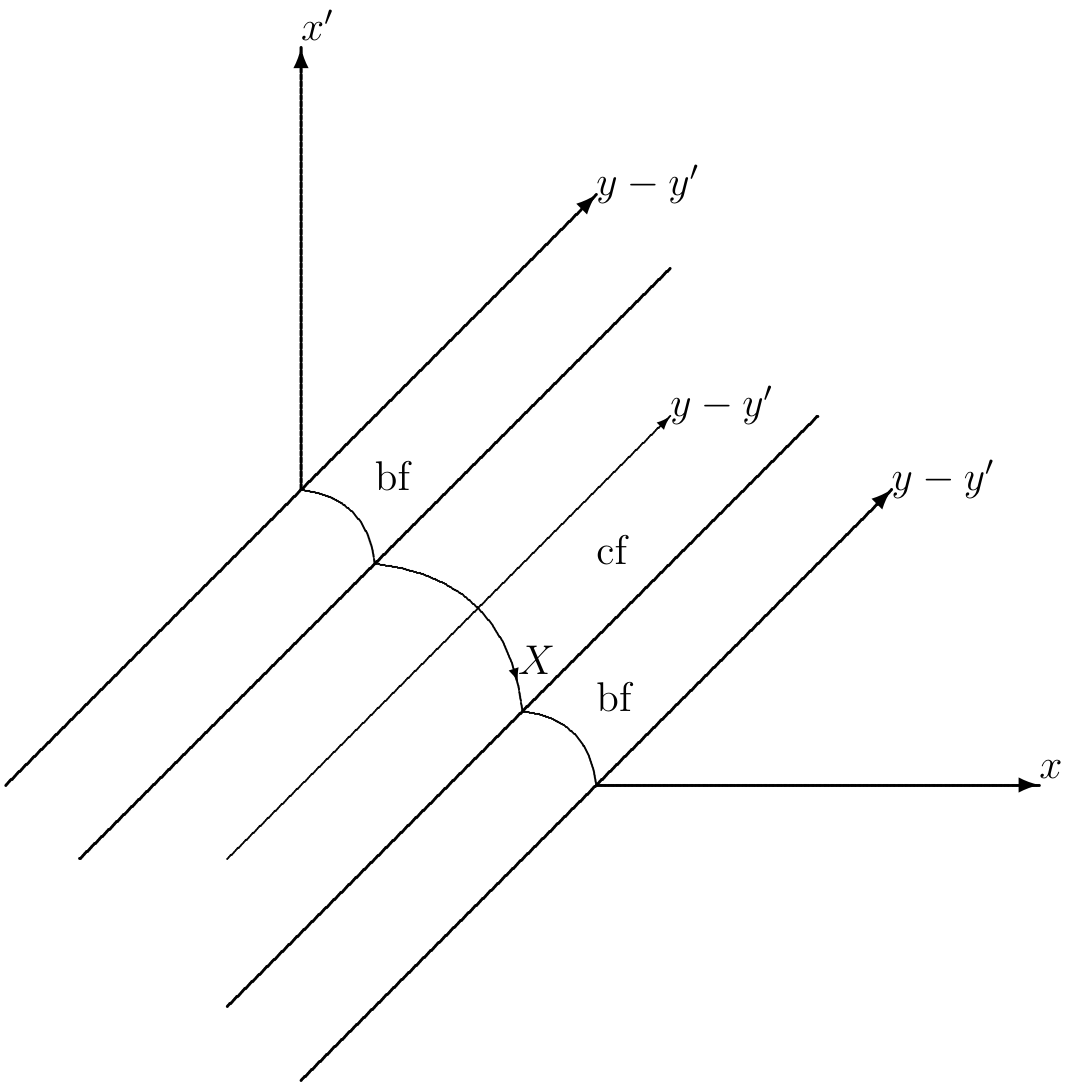}
\end{center}
\caption{The cusp double space as a blow-up of the b-double space. The
b font face is denoted by $\mathrm{bf}$, while the cusp front face is denoted by $\mathrm{cf}$.}
\label{fig:cusp-double}
\end{figure}

The double space is obtained from $M^2=M\times M$ by first performing
the b-blow up, i.e.\ blowing up $(\pa M)^2$, as in the
scattering setting, and then blowing up the lift of $x=x'$ at
$x=0$; see Figure~\ref{fig:cusp-double}. In valid coordinates on the b-double space near the lift of
$x=x'$, thus in the interior of the b-face, coordinates are
$x,y,\frac{x-x'}{x},y'$, and the lift of $x=x'$ is this
$\frac{x-x'}{x}=0$. The result of this blow-up is to obtain
coordinates
$$
x,y,X=\frac{x-x'}{x^2},y'
$$
near the interior of the front face.

Here one needs to check that the submanifold being blown up is
independent of the choice of $x$ modulo $O(x^2)$. 
We now show that if we define a new boundary defining function
$\tilde x = x + x^2 \phi$ where $\phi$ is a smooth function, then this
submanifold is unchanged.
Pulling back $\tilde x$ from the left and right
factors to the b-double space to get $\tilde x$ and $\tilde x'$,
$$
\tilde x = x + x^2 \phi(x,y) \qquad \tilde x' = x' + (x')^2 \phi(x',y'),
$$
so we can compute $\frac{x-x'}{x}$, the relevant b-front face coordinate in
the new variables, to obtain
\begin{equation*}\begin{aligned}
&\frac{\tilde x -\tilde x'}{\tilde x} = \frac{x - x' +
  x^2 \phi(x,y) - (x')^2 \phi(x',y')}{x (1+ x \phi(x,y))}\\
&= \frac{x-x'}{x} (1 + x \phi(x,y))^{-1} + x\frac{\phi(x,y) -
  \frac{(x')^2}{x^2} \phi(x',y')}{1 + x \phi(x,y)}\\
&= \frac{x-x'}{x} (1 + x \phi(x,y))^{-1} + x\frac{\phi(x,y) -
  (1-\frac{x-x'}{x})^2 \phi(x',y')}{1 + x \phi(x,y)}.
\end{aligned}\end{equation*}
Hence (since the blow-down maps are smooth, so
$x',y'$ are also smooth on the b-double space) $\frac{\tilde x -\tilde x'}{\tilde x}$ is smooth
on the b-double space and its zero set at $x=0$ is exactly the same
as that of $\frac{x-x'}{x}$. This means that the blow up creating the new double
space produces the same space as we change from $x$ to $\tilde x$ in
the definition, and thus it is well-defined independent of such choices.

A cusp pseudodifferential operator of order $m,\ell$ then has a Schwartz kernel that is
well-behaved on this double space in the sense that it is conormal to
the new, cusp, front face away
from the lifted diagonal, $\{X=0,y=y'\}$, vanishes to infinite order
at all boundary faces except the cusp front face, is conormal up
to the front face of order $\ell$, and is conormal to the diagonal of
order $m$. In particular, in a neighborhood of the diagonal it is
given by an oscillatory integral
$$
(2\pi)^{-n}\int e^{i(X\xicu+(y-y')\etacu)}a(x,y,\xicu,\etacu)\,d\xicu\,d\etacu
$$
relative to the density $\frac{|dx'\,dy'|}{(x')^{2}}$, where $a$ is a
symbol of order $m,\ell$, i.e.
$$
|\pa_x^j\pa_y^\alpha\pa_{\xicu}^k\pa_{\etacu}^\beta
a(x,y,\xicu,\etacu)|\leq Cx^{-\ell}\langle (\xicu,\etacu)\rangle^{m-k-|\beta|}.
$$

While, unlike the scattering algebra, the cusp algebra cannot be reduced modulo operators with
Schwartz Schwartz kernels to a H\"ormander algebra, in a somewhat
weaker sense, that still captures the near diagonal behavior, it can.
In this case the correspondence is with a different
H\"ormander algebra \cite{Hormander:v3} on $\RR^n$. Namely, taking symbols $a\in S_\infty^{m,l}$ on $\RR^n\times\RR^n$ such that
$$
|\pa_z^\alpha\pa_\zeta^\beta a(z,\zeta)|\leq C_{\alpha\beta} \langle
z\rangle^{l}\langle z_n\rangle^{-\alpha_n}\langle \zeta\rangle^{m-|\beta|},
$$
so the difference with the scattering case is that one only gains
$z_n$ decay upon differentiation in $z_n$ (and no decay otherwise),
and defining the Schwartz kernel of the standard, say, left quantization,
\begin{equation}\label{eq:cusp-on-Rn}
K_A(z,z')=(2\pi)^{-n}\int e^{i(z-z')\cdot\zeta}a(z,\zeta)\,d\zeta
\end{equation}
relative to the density $|dz'|$, the cusp algebra is obtained in open sets
of the form $O\times O\subset M^2$, where $O\subset M$ is identified with a
similar open set (with compact closure, if desired) in $\overline{\RR}\times\RR^{n-1}$ via a diffeomorphism by pulling back the Schwartz
kernel of an operator given by the just described left
quantization. Notice that $\overline{\RR}\times\RR^{n-1}$ corresponds
to a `cylindrical end' perspective on $\RR^n$, in which one
coordinate, say the last one, is distinguished, and the remaining ones
are required to stay in a bounded set. The reason this only captures the diagonal, thus
differential order, behavior of the cusp algebra is that the cusp
front face is global, i.e.\ includes points far from the diagonal in
$M^2$. Indeed, we can remedy this by treating the off-diagonal
behavior on $M^2$ by considering two disjoint open sets $O,U$ in $M$, mapping
them to disjoint open sets in $\overline{\RR}\times\RR^{n-1}$, and pulling back the
H\"ormander algebra Schwartz kernel from there. This perspective on
the cusp algebra was
explained in \cite{MR3792086}, and the equivalence is easily seen for
$x=z_n^{-1}$ can be taken to be the coordinate near infinity in
$\overline{\RR}$, so with $y=(z_1,\ldots,z_{n-1})$,
$\etaocu=(\zeta_1,\ldots,\zeta_{n-1})$, the phase function in
\eqref{eq:cusp-on-Rn} can be written as
$$
y\cdot \etaocu+(x^{-1}-(x')^{-1})\zeta_n,
$$
and
$$
x^{-1}-(x')^{-1}=-\frac{x-x'}{xx'}=-\frac{x}{x'}X=-(1-xX)^{-1}X,
$$
which (or
better yet, whose negative) could
have equally well been used in the definition of the cusp algebra
above. Note that the correspondence is $\zeta_n=-\xiocu$, and the
regularity of the amplitude $a$ is in terms of $x\pa_x,\pa_y$, which
equivalently means $z_n\pa_{z_n},\pa_{z_j},j=1,\ldots,n-1$, in the
relevant region, $|z_j|<C$, $z_n>1$.

Cusp pseudodifferential operators form a bi-filtered *-algebra as well
under adjoints and composition, however, the principal symbol only
captures the leading order behavior in the differential sense, thus is
insufficient to capture compactness of operators on the corresponding
cusp Sobolev spaces $\Hcu^{s,r}$. Indeed, these claims are
immediate from the just-described connection with the H\"ormander
algebra, and were proved by Mazzeo and Melrose in \cite{Mazzeo-Melrose:Fibred} using
geometric microlocal techniques.

\begin{prp}
If $A \in \Psicu^{m, \ell}$ with principal symbol, modulo
$\Psicu^{m-1,\ell}$, $[a]\in S^{m,\ell}/S^{m-1,\ell}$ and $B\in \Psicu^{m',\ell'}$ with principal symbol $[b]$ then $A \circ B \in \Psicu^{m + m', \ell + \ell'}$  with principal symbol $[a][b]=[ab]$. 
\end{prp}

\begin{prp}
If $A\in \Psicu^{m,\ell}$ with principal symbol $[a]$ is
elliptic, i.e.\ for some $c>0$,
\[ |a(x,y,\xicu, \etacu)| \ge c \langle \xicu,\etacu\rangle ^m x ^{-\ell} \text{ for } |(\xicu, \etacu)|\gg 1\]
then there is a parametrix $B\in \Psicu^{-m,-\ell}$ with error $AB-I,BA-I\in\Psicu^{-\infty, 0}$.
\end{prp}

\begin{prp}
If $A\in \Psicu^{m,\ell}$, then $A: \Hcu^{s,r} \to \Hcu^{s - m, r - \ell}$.
\end{prp}

Mazzeo and Melrose \cite{Mazzeo-Melrose:Fibred} define a normal operator to improve on this last
result and thus obtain compact errors, but we shall not need this
since our new algebra will have properties more akin to those of the
scattering algebra.

\subsection{The 1-cusp double space and the cusp pseudodifferential operators}
The simplest way to obtain the 1-cusp double space is from the cusp
one by blowing up the boundary of the lifted diagonal
$\{X=0,y-y'=0\}$, i.e.\ $\{X=0,y-y'=0,x=0\}$; see Figure~\ref{fig:1cusp-double}. Notice that the lifted
diagonal indeed only intersects the cusp front face (in particular
does not intersect the b-front face), so local coordinates in the
interior of the cusp front face can be used. Since this submanifold is purely
geometric, it does not depend on any additional information beyond
what went into the definition of the cusp double space, namely the
boundary defining function defined up to $O(x^2)$ terms. Concretely,
in a neighborhood of the interior of the front face $x$ is relatively
large, and we thus obtain coordinates
$$
x,y,V=\frac{X}{x}=\frac{x-x'}{x^3},Y=\frac{y-y'}{x}.
$$
The Schwartz kernels of our new operators then are required to be
well-behaved on the new double space in the sense that they are
conormal to the new, 1-cusp, front face
away from the lifted diagonal, $\{V=0,\ Y=0\}$, vanish to infinite
order at all boundary faces, are conormal to the 1-cusp face of order
$\ell$ and to the lifted diagonal of order $m$. In particular, in a
neighborhood of the lifted diagonal they are given by an oscillatory
integral
\begin{equation}\label{eq:SK-osc-int}
  K_A(x,y,V, Y) = 
(2\pi)^{-n}\int e^{i(V\xiocu+Y\etaocu)}a(x,y,\xiocu,\etaocu)\,d\xiocu\,d\etaocu
\end{equation}
relative to the density $\frac{|dx'\,dy'|}{(x')^{n+2}}$, which arises
from Jacobian factors caused by the blow-ups and which we explain
below, where $a$ is a `product type'
symbol of order $m,\ell$, i.e.
\begin{equation}\label{eq:1c-symbol-est}
|\pa_x^j\pa_y^\alpha\pa_{\xiocu}^k\pa_{\etaocu}^\beta
a(x,y,\xiocu,\etaocu)|\leq Cx^{-\ell}\langle (\xiocu,\etaocu)\rangle^{m-k-|\beta|}.
\end{equation}
On the other hand, away from the lifted diagonal but near the 1-cusp front
face the Schwartz kernel satisfies estimates
\begin{equation}\label{eq:conormal-to-ff}
|V^i Y^\gamma (x\pa_x)^j\pa_y^\alpha \pa_V^k\pa_Y^\beta K_A(x,y,V,Y)|\leq Cx^{-\ell},
\end{equation}
with $C$ depending on the indices $i,j,k,\alpha,\beta,\gamma$,
where $i,j,k\in\NN$, $\alpha,\beta,\gamma\in\NN^{n-1}$, which also encodes,
via the powers of $V,Y$, the rapid decay to the cusp front face near
the corner. Notice that as all the ingredients of the definition are
diffeomorphism invariant, so is the class of 1-cusp pseudodifferential
operators.

\begin{figure}[ht]
\begin{center}
\includegraphics[width=110mm]{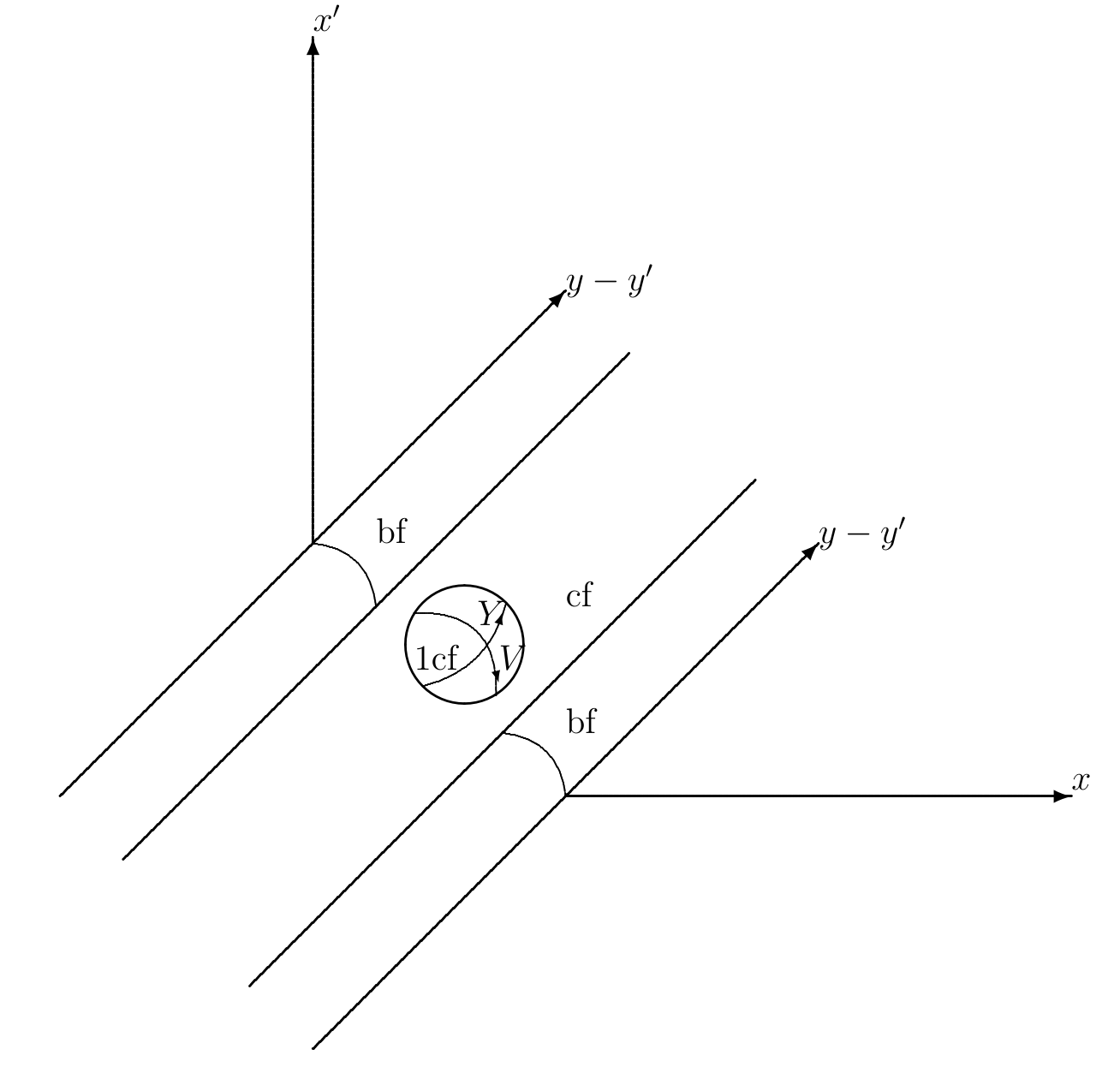}
\end{center}
\caption{The 1-cusp double space as a blow-up of the cusp double space. The
b font face is denoted by $\mathrm{bf}$, the cusp front face is
denoted by $\mathrm{cf}$, while the 1-cusp front face by $\mathrm{1cf}$.}
\label{fig:1cusp-double}
\end{figure}

While strictly speaking, by the definition of conormal distributions, \eqref{eq:SK-osc-int} is to be interpreted as
a local oscillatory integral, i.e.\ valid with $V,Y$ bounded, one {\em
  can} interpret it more globally. The reason is that by the basic
properties of the Fourier transform, outside $V=0$, $Y=0$, it produces
a Schwartz function in $(V,Y)$ with values in conormal functions of
$(x,y)$, i.e.\ \eqref{eq:conormal-to-ff} holds for the right hand side
of \eqref{eq:SK-osc-int} regardless of $m,\ell$.

We can shed some light on this algebra by also relating its double
space to that of the
scattering algebra; this relation is of some importance since the
1-cusp algebra itself arises for us in the setting of an asymptotically conic metric,
which is naturally described by, and in particular has its
bicharacteristics described by, the scattering geometry. Namely, for
this perspective, within the scattering double space, one blows up
the lift of $x=x'$, i.e.\ in local coordinates $\{x=0,X=0\}$,
intersected with
the scattering front face, $x=0$. In the interior of the new front
face this indeed produces local coordinates
$$
x,y,V=\frac{X}{x},Y,
$$
matching those of the blow-up obtained from the cusp algebra, and
establishing a natural local (in the region of validity of the two
coordinates) diffeomorphism between the two spaces. A
subtlety here, however, is that, unlike for the cusp approach above, the manifold we blow up intersects
faces other than the scattering front face as well, namely the
boundary of the scattering front face, so for a full discussion from
this perspective valid coordinates must also be described and used in those
regions. Another potential issue, which however is easily resolved
using a straightforward
modification of the above computation that the cusp algebra is well-defined, is that we need to check that the
submanifold being blown up is well defined if $x$ is only well-defined
up to adding $O(x^2)$ terms. In any
case, this well-definedness statement follows from the just established diffeomorphism, at least in the
interior of the new front faces.
The figure below represents this new double-space.

\begin{figure}[ht]\begin{center}
    
\includegraphics[width=12cm]
{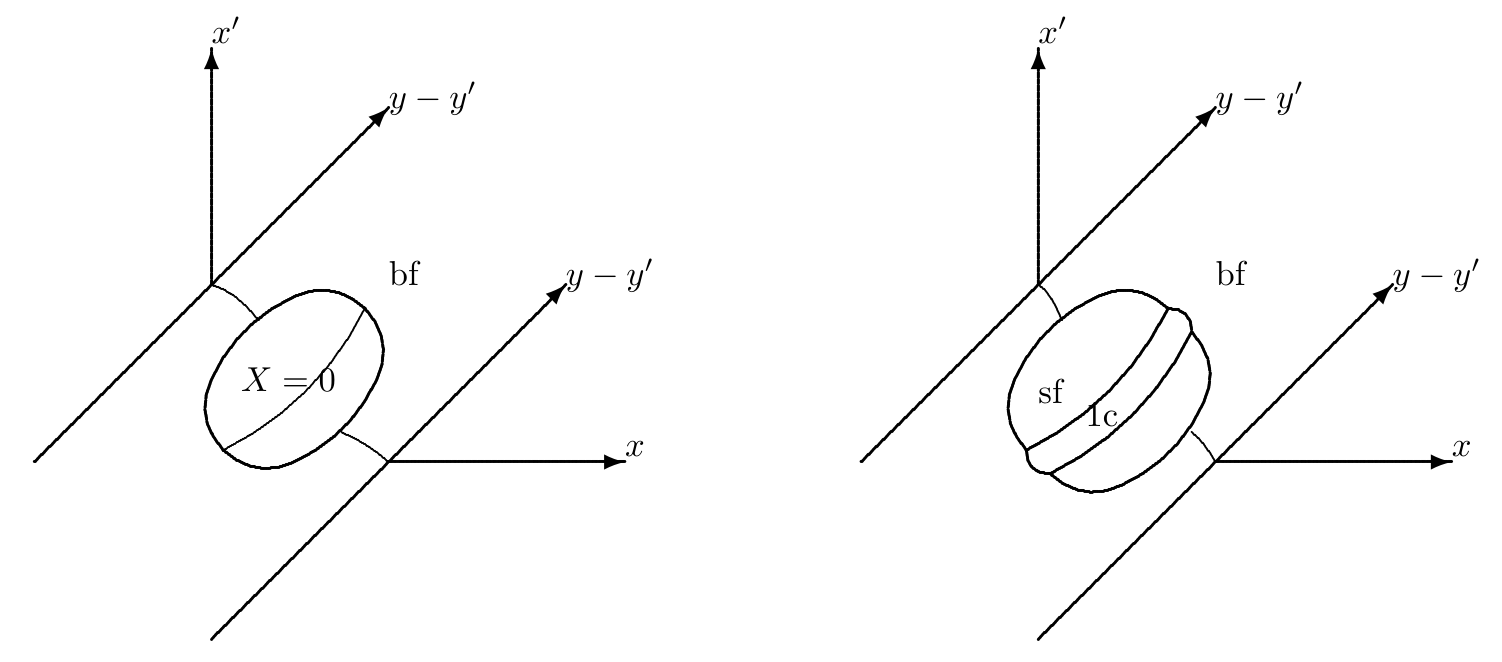}
\end{center}
\caption{1-cusp double space by a
    blow-up of the scattering double space: on the left the
    submanifold $X=0$ of the scattering front face, $\mathrm{sf}$, and
    on the right the resulting resolution, with the front face
    labelled by $\mathrm{1c}$. A neighborhood of the interior of the
    front face is naturally diffeomorphic to the front face
    $\mathrm{1cf}$ shown in Figure~\ref{fig:1cusp-double} in the sense
    that the identity map between the interiors of $M^2$ smoothly
    extends up to these boundary faces, but the
    boundary of the two front faces is quite different.}\label{hmm}  \end{figure}

The shaded portion in Figure~\ref{hmm} corresponds to the new front
face, and $V$ and $Y$ are coordinates for this new front face (along
with $y$, not shown in this picture).

While the identity map in the interior does not induce a global
diffeomorphism between the space we just obtained and our 1-cusp double
space, for instance due to intersection of the 1-cusp front face
intersecting the b front face (unlike from the cusp approach), the
space of conormal distributions, conormal to the diagonal and the
front face, and vanishing to infinite order at every other face, is the
same. Thus, we can consider the 1-cusp pseudodifferential operators
both in relation to the cusp ones and to the scattering ones.

This definition is thus analogous to the geometric definition of
scattering and cusp pseudodifferential operators earlier. As with
scattering and cusp operators, smoothing
operators (elements of $\Psiocu^{-\infty,\ell}$) have Schwartz kernels without a conormal singularity along
the diagonal, and residual operators (elements of $\Psiocu^{-\infty,-\infty}$) have Schwartz kernels which
additionally vanish to infinite order on the new front face. Therefore
residual operators have Schwartz kernels which are also residual on
the blown-down scattering or cusp space, so that they are residual
scattering and cusp
operators.

As already mentioned, if $A\in\Psiocu^{-\infty,\ell}$, $K_A$ satisfies
\eqref{eq:conormal-to-ff}, say near the interior of the cusp face, i.e.\
where $x|V|,x|Y|<C$, and if $v$ is conormal of order $r$, i.e.\
$|(x\pa_x)^j\pa_y^\alpha v|\leq C_{j\alpha} x^{-r}$, then
\begin{equation*}\begin{aligned}
(Av)(x,y)&=\int
K_A\Big(x,y,\frac{x-x'}{x^3},\frac{y-y'}{x}\Big)v(x',y')\,\frac{dx'\,dy'}{(x')^{n+2}}\\
&=\int K_A(x,y,V,Y) v(x-x^3V,y-xY)\,dV\,dY,
\end{aligned}\end{equation*}
which is immediately seen (by applying products of $x\pa_x$ and
$\pa_y$, and using the rapid decay of $K_A$ in $V,Y$) to be conormal of order $r+\ell$ (and indeed
$\CI$ if $K_A$ and $v$ are actually smooth), i.e.\ the orders add up; the Jacobian in the
change of variables is the reason for the normalization of the density
that we have adopted.

Just as in the scattering and cusp cases, there is a simple way to
reduce this to an algebra in $\RR^n$. In this case
$$
x,y,\tilde V=\frac{1}{x^2}-\frac{1}{(x')^2},\ \tilde Y=\frac{y}{x}-\frac{y'}{x'}
$$
also give valid coordinates in the interior of the 1c-front
face. Indeed,
$$
\tilde V=\frac{(x'-x)(x'+x)}{x^2(x')^2}=-V\frac{x(x'+x)}{(x')^2}=-V\frac{2-x^2V}{(1-x^2V)^2}
$$
as $\frac{x'}{x}=1-x^2V$, and
$$
\tilde Y=\frac{y-y'}{x}+\Big(\frac{1}{x'}-\frac{1}{x}\Big)y'=Y+xV(1-x^2V)^{-1}(y-xY),
$$
and indeed at $x=0$, we have $\tilde V=-2V$, $\tilde Y=Y$; the
converse direction is similar. Thus, we can equivalently write in
place of \eqref{eq:SK-osc-int}
\begin{equation}\label{eq:SK-osc-int-tilde}
  K_A(x,y,\tilde V, \tilde Y) = {(2\pi)^{-n}}\int e^{i \widetilde{\xiocu} \tilde V
    + \widetilde{\etaocu} \cdot \tilde Y}\tilde a(x,y,\widetilde{\xiocu},
  \widetilde{\etaocu})\, d\widetilde{\xiocu}\, d\widetilde{\etaocu},
\end{equation}
with $\tilde a\in S^{m,l}$, i.e.\ satisfying the same kinds of
`product type' symbol
estimates \eqref{eq:1c-symbol-est}.
But then with $z_n=1/x^2$, $z_j=y_j/x$ ($j=1,\ldots,n-1$), so $x=z_n^{-1/2}$,
$y_j=z_j/z_n^{1/2}$, in the region where $(x,y_j)$ is bounded, i.e.\
$z_n$ bounded away from $0$ and $|z_j|<Cz_n^{1/2}$, this amounts to
exactly an oscillatory integral of the form
\begin{equation}\label{eq:SK-osc-int-Rn}
 {(2\pi)^{-n}}\int e^{i (\zeta_n (z_n-z_n')
    + \sum_{j=1}^{n-1}\zeta'_j(z_j-z'_j))}\tilde a(z,\zeta)\, d\zeta,
\end{equation}
relative to the density $|dz'|=2\frac{|dx'\,dy'|}{(x')^{n+2}}$,
with $\tilde a$ well-behaved (conormal) in terms of $z_n^{-1/2}$,
$z_j/z_n^{1/2}$, i.e.\ in a parabolic compactification, with the
relevant region being $z_n>1$, $|z_j|<Cz_n^{1/2}$,
$j=1,\ldots,n-1$. Concretely, as
$$
x\pa_x=-2z_n\pa_{z_n}-\sum_{j=1}^{n-1}z_j\pa_{z_j}=-2z_n\pa_{z_n}-\sum_{j=1}^{n-1}\frac{z_j}{z_n^{1/2}}z_n^{1/2}\pa_{z_j}
$$
and
$$
\pa_{y_j}=z_n^{1/2}\pa_{z_j},
$$
this means that, in this region, iterated regularity of $\tilde a$
with respect to $x\pa_x$ and $\pa_{y_j}$ is equivalent to that with
respect to
$z_n\pa_{z_n}$ and $z_n^{1/2}\pa_{z_j}$. One
can instead work globally (using $\langle z\rangle$ in place of $z_n$) with symbol estimates
\begin{equation}\label{eq:Rn-symbol-estimates}
|\pa_z^\alpha\pa_\zeta^\beta \tilde a(z,\zeta)|\leq
C_{\alpha\beta}\langle\zeta\rangle^{m-|\beta|}\langle z\rangle^{\ell/2-|\alpha|/2-\alpha_n/2}.
\end{equation}
Here the $\ell/2$ in the power of $\langle z\rangle$ arises from
$z_n=x^{-2}$ in the relevant region, so this power is locally
equivalent to $x^{-\ell}$.
We define \eqref{eq:Rn-symbol-estimates} as the parabolic symbol class
$S^{m,\ell/2}_{\para}$, and write the corresponding pseudodifferential
operators, via \eqref{eq:SK-osc-int-Rn} as $\Psipara^{m,\ell/2}$.
Notice that the basis for 1-cusp vector  fields (over $\CI$ functions
of $x,y$, or equivalently of $z_n^{-1/2}$,
$z_j/z_n^{1/2}$) is
\begin{equation}\begin{aligned}\label{eq:vfs-ocu-Rn}
x^3\pa_x
&=-2\pa_{z_n}-\sum_{j=1}^{n-1}\frac{z_j}{z_n^{1/2}}z_n^{-1/2}\pa_{z_j},\
x\pa_{y_j}=\pa_{z_j},\ j=1,\ldots,n-1,
\end{aligned}\end{equation}
which is equivalent to $\pa_{z_1},\ldots,\pa_{z_n}$.

\subsection{Algebraic properties}\label{sec:algebra}
Given the identification of $\Psiocu^{m,\ell}$ locally with the
pseudodifferential operators $\Psipara^{m,\ell/2}$ on $\RR^n$, the
algebra properties of $\Psiocu$ follow immediately from those
of $\Psipara$. The latter in turn are immediate with the standard
composition, etc, formulae on $\RR^n$, applicable even in H\"ormander's algebra
$\Psi_\infty$ (with just uniform $z$ estimates, without decay on
differentiation). Note that the principal symbol of $A \in
\Psiocu^{m, \ell}$ needs to be understood modulo $S^{m-1,\ell-1}$, for
this corresponds to the statement that for $\tilde
A\in\Psipara^{m,\ell/2}$, the principal symbol is in
$S^{m,\ell/2}/S^{m-1,\ell/2-1/2}$ in view of the defining estimate and
the standard symbol expansion.

\begin{prp}\label{prop:1c-comp}
If $A \in \Psiocu^{m, \ell}$ with principal symbol, modulo
$\Psiocu^{m-1,\ell-1}$, $[a]\in S^{m,\ell}/S^{m-1,\ell-1}$ and $B\in \Psiocu^{m',\ell'}$ with principal symbol $[b]$ then $A \circ B \in \Psiocu^{m + m', \ell + \ell'}$  with principal symbol $[a][b]=[ab]$. 
\end{prp}

As usual, this implies that there is a parametrix for elliptic
operators:

\begin{prp}\label{prop:1c-param}
If $A\in \Psiocu^{m,\ell}$ with principal symbol $a$ is
elliptic, i.e.\ for some $c>0$,
\[ |a(x,y,\xiocu,\etaocu)| \ge C \langle \xiocu,\etaocu\rangle ^m x ^{-\ell} \text{ for } |(\xiocu,\etaocu)|\gg 1 \text{ or } x \ll 1\]
then there is a parametrix $B\in \Psiocu^{-m,-\ell}$ with error in $\Psiocu^{-\infty, -\infty}$.
\end{prp}

The positive integer order 1-cusp Sobolev spaces can be defined via regularity
with respect to the 1-cusp vector fields, i.e.\ $u\in \Hocu^s$ if $u\in
L^2$ (relative to a 1-cusp density, $\frac{|dx\,dy|}{x^{n+2}}$), and
$V_1\ldots V_j u\in L^2$ as well if $j\leq s$ and $V_j\in\Vocu$; the
negative integer ones then can be defined via duality. The
weighted spaces $\Hocu^{s,r}$ are $x^r\Hocu^s$. Equivalently, we can say, for
any real $s,r$, writing $\cS'$ for the space of tempered distributions
on $M$, i.e.\ the dual of $\CI$ functions vanishing to infinite order
at $\pa M$,
$$
\Hocu^{s,r}=\{u\in\cS':\ \exists A\in\Psiocu^{s,r}\ \text{elliptic}\
\text{and}\ Au\in L^2\}.
$$
In view of the identification of the 1-cusp vector fields with those
on $\RR^n$,
\eqref{eq:vfs-ocu-Rn}, respectively that of the pseudodifferential
algebras, on $\RR^n$ these Sobolev spaces correspond to the standard
weighted Sobolev space $H^{s,r/2}$, and the following mapping result
is immediate.

\begin{prp}
If $A\in \Psiocu^{m,\ell}$, then $A: \Hocu^{s,r} \to \Hocu^{s - m, r - \ell}$. 
\end{prp}

We comment on a different way of analyzing
the 1-cusp algebra by relating it to the cusp algebra at a symbolic level. This
relationship is exactly the same as that of the scattering (which is
1-b from this perspective) and b-algebras, as explained in the second
microlocal discussion of \cite[Section~5]{Vasy:Zero-energy}, thus we
will be brief.

Concretely,
the Schwartz kernels of the cusp operators are, near the diagonal,
given by oscillatory integrals of the form
$$
(2\pi)^{-n}\int e^{i(X\xicu+(y-y')\etacu)}a^{\cul}(x,y,\xicu,\etacu)\,d\xicu\,d\etacu,
$$
relative to the density $\frac{|dx'\,dy'|}{(x')^2}$,
while those of the 1-cusp ones have the form
$$
(2\pi)^{-n}\int e^{i(V\xiocu+Y\etaocu)}a^{\ocul}(x,y,\xiocu,\etaocu)\,d\xiocu\,d\etaocu,
$$
relative to the density $\frac{|dx'\,dy'|}{(x')^{n+2}}$. These are the
same, however, if we write $\xiocu=x\xicu$, $\etaocu=x\etacu$, up to
an overall factor of $(x/x')^n$ (which is irrelevant for the class,
and is identically 1 on the cusp front face). This simply corresponds
to the natural coordinates on the cotangent bundles:
$$
\xicu\,\frac{dx}{x^2}+\etacu\cdot dy=\xiocu\, \frac{dx}{x^3}+\etaocu\cdot \frac{dy}{x}.
$$
This means that geometrically one can (almost, as we explain) obtain the 1-cusp symbol space
from the cusp one by blowing up the corner of the fiber-compactified
cusp cotangent bundle; see Figure~\ref{fig:cusp-2micro}. Indeed, coordinates there, where say $|\xicu|$ is
large relative to $|\etacu|$, are
$$
x,y,|\xicu|^{-1},\frac{\etacu}{|\xicu|},
$$
the corner is $x=0, |\xicu|^{-1}=0$, so in the region where $x$ is
relatively large (relative to the other defining function,
$|\xicu|^{-1}=0$, of the submanifold being blown up), on the blown up space the coordinates become
$$
x,y, |\xicu|^{-1}/x,\frac{\etacu}{|\xicu|},
$$
which are exactly
$$
x,y,|\xiocu|^{-1}, \frac{\etaocu}{|\xiocu|},
$$
i.e.\ coordinates near the corner of the fiber-compactified 1-cusp
cotangent bundle. The `almost' in the identification refers to  the
region where $|\xicu|^{-1}$ is relatively large, where valid
coordinates are
$$
x/|\xicu|^{-1},y, |\xicu|^{-1},\frac{\etacu}{|\xicu|},
$$
which are
$$
|\xiocu|,y, x|\xiocu|^{-1},\frac{\etaocu}{|\xiocu|}.
$$
But these are valid coordinates near the corner if one blows up the
boundary of the zero section of the 1-cusp cotangent bundle. All
remaining regions are handled similarly, cf.\
\cite[Section~5]{Vasy:Zero-energy}.

\begin{figure}[ht]
\begin{center}
\includegraphics[width=40mm]{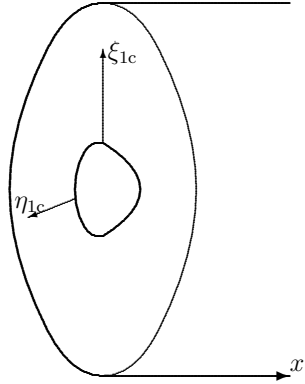}
\end{center}
\caption{The resolution of the corner in the compactified cusp
  cotangent bundle, i.e.\ the symbol space. The curved
  boundary hypersurface in the center at $x=0$ is the original
  boundary fiber of the cusp cotangent bundle. With the corner blown
  up, one obtains a new front face which is naturally diffeomorphic to
the fiber of the 1-cusp cotangent bundle blown up at its zero section.}
\label{fig:cusp-2micro}
\end{figure}

In particular, 1-cusp pseudodifferential symbols can be considered as
cusp ones well-behaved on the corner-blown up cusp space. But blowing
up the corner of a manifold with corners does not change the space of
conormal functions. Thus, much as in the scattering-b relation of
\cite[Section~5]{Vasy:Zero-energy}, we can consider the 1-cusp
operator symbols as conormal (non-classical) cusp operators symbols,
and simply apply the results from the cusp algebra. Concretely,
$$
\Psiocu^{m,\ell}\subset \Psicu^{m,\max(\ell,\ell-m)}\cap\Psicu^{\max(m,0),\ell},
$$
which in the `base case' of $\Psiocu^{0,0}$ is simply the statement
$\Psiocu^{0,0}\subset\Psicu^{0,0}$.

In fact, it is better yet to consider the 3-ordered second microlocal
algebra with 3 orders from the resolved cusp cotangent bundle perspective:
lifted fiber infinity, the lifted boundary and the front face. The
resulting space is then a tri-filtered *-algebra by the cusp results,
and $\Psiocu^{m,\ell}\subset\Psiocucu^{m,\ell,\ell}$.

Now, in order to deal with $\Psiocu^{m,\ell}$ itself in this manner
without the previous discussion, we need to `blow
down' the cusp face in the 2-microlocalized cusp cotangent bundle,
i.e.\ show that the algebraic properties descend to the `blown down
space', namely the 1-cusp cotangent bundle, so that the composition of
elements of $\Psiocu^{m,\ell}$ and $\Psiocu^{m',\ell'}$ is not merely
in the 2-microlocal algebra, but in $\Psiocu^{m+m',\ell+\ell'}$
itself. But modulo $\Psiocu^{-\infty,\ell}$, we can write elements of
$\Psiocu^{m,\ell}$ as quantizations of symbols supported in
$|(\xiocu,\etaocu)|\geq 1$, and thus in the 2-microlocal space away
from the cusp face. The cusp composition results imply that the
composition of any two
so-microlocalized elements of the cusp algebra in fact results in a
similarly microlocalized element thus yielding an element of
$\Psiocu^{m+m',\ell+\ell'}$, and one then only needs to prove
composition results for smoothing operators, from which we refrain here.

\subsection{The semiclassical version of the algebra}\label{sec:semicl-alg}
In order to make the errors in the elliptic parametrix construction
not just compact but small we also need a semiclassical version of the
algebra, and indeed a semiclassical foliation algebra. In the standard
and scattering settings the latter was introduced in
\cite{Vasy:Semiclassical-X-ray} for the same reason. Here the
foliation is given by level sets of the boundary defining function
$x$, so now we regard $x$ as fixed (not just up to adding $O(x^2)$ terms).

In the $\RR^n$ version, the semiclassical foliation quantization takes
the form
\begin{equation}\label{eq:SK-osc-int-Rn-hh}
 {(2\pi)^{-n}}h^{-n/2-1/2}\int e^{i (\zeta_n (z_n-z_n')/h
    + \sum_{j=1}^{n-1}\zeta'_j(z_j-z'_j)/h^{1/2})}\tilde a(z,\zeta)\, d\zeta,
\end{equation}
with symbols still satisfying
\begin{equation}\label{eq:Rn-symbol-estimates-hh}
|\pa_z^\alpha\pa_\zeta^\beta \tilde a_h(z,\zeta)|\leq
C_{\alpha\beta}\langle\zeta\rangle^{m-|\beta|}\langle z\rangle^{\ell/2-|\alpha|/2-\alpha_n/2}.
\end{equation}
This could be regarded as a standard semiclassical quantization, i.e.\
where both powers of $h$ in the exponent are $1$ and the overall
pre-factor is $h^{-n}$, with a worse behaved symbol, but the present version gives more precise
algebraic properties. This is transferred over to the manifold $M$ as
in the non-semiclassical setting, ensuring that in $M^2$ away from the
diagonal the Schwartz kernel not only vanishes to infinite order at
the boundary, like in the 1-cusp case, but also to infinite order in
$h$. With this definition the standard composition results on $\RR^n$
yield the following results.

\begin{prp}
If $A \in \Psiocuhh^{m, \ell}$ with principal symbol, modulo
$h^{1/2}\Psiocuhh^{m-1,\ell-1}$, $[a]\in S^{m,\ell}/h^{1/2} S^{m-1,\ell-1}$ and $B\in \Psiocuhh^{m',\ell'}$ with principal symbol $[b]$ then $A \circ B \in \Psiocuhh^{m + m', \ell + \ell'}$  with principal symbol $[a][b]=[ab]$. 
\end{prp}

Here the gain of $h^{1/2}$ in the symbol algebra corresponds to the
$h^{1/2}$ appearing with the foliation tangent variables, $y$ or
$z_j$, $j=1,\ldots,n-1$.
As usual, this implies that there is a parametrix for elliptic
operators:

\begin{prp}
If $A\in \Psiocuhh^{m,\ell}$ with principal symbol $a$ is
elliptic, i.e.\ for some $c>0$,
\[ |a(x,y,\xiocu,\etaocu)| \ge C \langle \xiocu,\etaocu\rangle ^m x
  ^{-\ell} \text{ for } |(\xiocu,\etaocu)|\gg 1 \text{ or } x \ll 1
  \text{ or } h\ll 1\]
then there is a parametrix $B\in \Psiocuhh^{-m,-\ell}$ with error in $h^\infty\Psiocuhh^{-\infty, -\infty}$.
\end{prp}

The triviality of the error, reflecting by the $h^\infty$ factor, is
what gives the smallness of the error for $h$ sufficiently small, and
thus the invertibility of $A_h$ in that case.

The positive integer order semiclassical foliation 1-cusp Sobolev
spaces are the same as the 1-cusp spaces, but with respect to an
$h$-dependent norm. They can be defined via regularity
with respect to the corresponding foliation semiclassical 1-cusp vector fields, $V\in\Vocuhh$,
which are of the form
$$
a_0(hx^3\pa_x)+\sum_{j=1}^{n-1}a_j(h^{1/2}x\pa_{y_j}).
$$
Thus, the norm is locally (and in general by a partition of unity) given by
$$
\|u\|_{\Hocuhh^s}^2=\|u\|^2_{L^2}+\sum_{j+|\alpha|\leq s}
\|(hx^3\pa_x)^j (h^{1/2}x\pa_{y_j})^\alpha u\|^2_{L^2};
$$
the
negative integer ones then can be defined via duality. The
weighted norms $\Hocuhh^{s,r}$ are those corresponding to $x^r\Hocuhh^s$. Equivalently, we can say, for
any real $s,r$, with $s,r\geq 0$, say,
$$
\|u\|_{\Hocuhh^{s,r}}^2=\|u\|^2_{L^2}+\|Au\|^2_{L^2},
$$
where $A\in\Psiocuhh^{s,r}$ is elliptic.
This gives

\begin{prp}
If $A\in \Psiocuhh^{m,\ell}$, then $A: \Hocuhh^{s,r} \to \Hocuhh^{s - m, r - \ell}$. 
\end{prp}

\section{Invertibility of the X-ray transform}\label{sec:normal-op}
Our main technical result concerns a modified normal operator for the X-ray
transform, namely
$$
A=e^{-\Phi/h}L\tilde\chi Ie^{\Phi/h},
$$
where
$$
(Lw)(z)=\int_{S_z M} w(z,v) |d\sigma_z(v)|,
$$
$\tilde\chi=\tilde\chi(z,\lambda/(h^{1/2} x),\omega)$,
$\Phi=\frac{-1}{2x^2}$ is a Gaussian weight, and $\sigma$ is a smooth
positive measure on $S_z M$, i.e.\ in $(\lambda,\omega)$. Here we take a
different normalization of $L$ than
\cite{Zachos:Thesis}, which introduces an additional
$x^{-1}$ factor in $L$, making the decay order $0$ (rather than $-1$) for $A$ in the
following theorem.

\begin{thm}\label{ispsdo}
The modified normal operator $A$ of $I$ is an operator in
$h\Psiocuhh^{-1, -1}$. Furthermore, for a suitable choice of $\tilde\chi$, its principal symbol is
elliptic in a collar neighborhood of $\pa M$ both in the sense of the
1-cusp algebra and semiclassically.
\end{thm}

The proof will take up
Sections~\ref{sec:geodesic-structure}-\ref{sec:X-ray-inv}. Then in
Section~\ref{sec:ispsdo-conseq} we derive some immediate consequences,
and then in Section~\ref{sec:combined-sc-1c} we introduce the
artificial boundary method in this context to prove Theorem~\ref{thm:main-conic-ops}.

\subsection{Structure of the geodesics}\label{sec:geodesic-structure}
In order to get started, first we need to describe the geodesics in
some detail.
They satisfy
Hamilton's equation of motion, i.e.\ have tangent vector given by the
Hamilton vector field $H_g$ of the dual metric function. The dual
metric function is a symbol on $\Tsc^*M$, of order
$(2,0)$, which is elliptic away from the 0-section. Thus, $H_g$ is of
the form $x$ times a b-vector field; this structure
is generally
the case of symbols of order $(2,0)$ on $\Tsc^*M$, as follows from
\cite{RBMSpec}. Indeed, as we already mentioned, explicitly
$$
\frac{1}{2}H_g=x\Big(\tau(x\pa_x+\mu\cdot\pa_\mu)-|\mu|^2\pa_\tau+\frac{1}{2}H_h+xV\Big),
$$
where $V\in\Vb(\Tsc^*M)$. Here and below, for convenience, we use a product
decomposition of a neighborhood of $\pa M$ respecting the preferred
boundary defining function $x$; the
concrete choice is irrelevant. Also, in this geometric discussion we
write covectors as
$$
\tau\frac{dx}{x^2}+\mu\cdot\frac{dy}{x}.
$$
Correspondingly,
it is useful to consider integral curves $\tilde\gamma$ of $\scH_g=x^{-1}H_g$. If the actual
bicharacteristics are $c$, then the relationship is via the
reparameterization of the integral curves via $\frac{dt}{ds}=x(c(s))$,
i.e.\ $\frac{ds}{dt}=x(\tilde\gamma(t))^{-1}$; thus the X-ray
transform can be rewritten in terms of $\tilde\gamma$, or its
projection $\gamma$ to the base manifold, as
\begin{equation}\label{eq:sc-reparameterized-Xray}
If(\gamma)=\int f(\gamma(t))\,x(\gamma(t))^{-1}\,dt,
\end{equation}
i.e.\ as a weighted X-ray transform. Later on we shall make a further
change of parameterization to deal with the Schwartz kernel of the
operator as $t\to\pm\infty$ (and thus $x(\gamma(t))\to 0$).

Along the integral curves $\tilde\gamma=\tilde\gamma(t)$ of
$\frac{1}{2}\scH_g$, $\frac{dx}{dt}=\tau x+x^2 f_1$, $f_1$ smooth, and
hence
$$
\frac{d^2 x}{dt^2}=(-|\mu|^2+\tau^2) x+x^2f_2,
$$
with $f_2$ smooth, so along the unit level set of the dual metric function,
\begin{equation}\label{eq:Hessian-x}
\frac{d^2 x}{dt^2}=(-2|\mu|^2+1) x+x^2f_2,
\end{equation}
which is negative, with an upper bound given by $x$ times a negative constant, where $|\mu|>3/4$
say, provided $x$ is sufficiently small. In particular, if $\frac{dx}{dt}=0$, so $\tau=O(x)$, $\frac{d^2
  x}{dt^2}<0$ showing the concavity of the level sets of $x$ from the
sublevel sets. In general we will work in a neighborhood $x<x_0$ of
infinity in which this concavity statement holds.

In fact, one can have even stronger convexity by working with
$x^{-2}$ in place of $x$ as $\frac{dx^{-2}}{dt}=-2x^{-2}(\tau+xf_1)$,
hence
$$
\frac{d^2 x^{-2}}{dt^2}=2x^{-2}(2\tau^2+|\mu|^2+xf_2),
$$
which is positive, with an $x^{-2}$ times a positive lower bound on the characteristic
set for $x$ small.

Now, $\scH_g$ being a b-vector field, it is a linear combination of
$x\pa_x$, $\pa_y$, $\pa_{\tau}$ and $\pa_{\mu}$.
In view of this, it is useful to write the tangent vector to the projected bicharacteristic $\gamma$, which is thus the
pushforward of $\scH_g$ to the base manifold, as a b-tangent vector,
$$
\lambda x\pa_x+\omega\pa_y.
$$
Notice that in fact $\lambda$ is independent of the choice of the
product decomposition respecting the preferred boundary defining
function $x$, i.e.\ if $(x',y')$ are other coordinates with $x'=x$,
then
$$
\lambda x\pa_x+\omega\pa_y=\lambda (x'\pa_{x'})+\omega'\pa_{y'}.
$$
Now writing the $x$, resp.\ $y$, component of $\gamma$ as
$\gamma^{(1)}$, resp.\ $\gamma^{(2)}$, as before, by the smoothness of
the flow and as it is tangent to $x=0$, so it preserves $x=0$, we have
$$
\gamma^{(1)}_{x,y,\lambda,\omega}(t)=x\tilde\Gamma^{(1)}(x,y,\lambda,\omega,t)
$$
with $\tilde\Gamma^{(1)}$ smooth as
$$
\gamma^{(1)}_{x,y,\lambda,\omega}(t)=x\int_0^1\pa_1
\gamma^{(1)}_{\sigma x,y,\lambda,\omega}(t)\,d\sigma,
$$
with $\pa_1$ denoting derivative in the first subscript slot,
and
$$
\tilde\Gamma^{(1)}(x,y,\lambda,\omega,0)=1,\ \tilde\Gamma^{(1)}(0,y,\lambda,\omega,t)=\pa_1 \gamma^{(1)}_{0,y,\lambda,\omega}(t).
$$
Now, $\gamma$ having tangent vector
$\lambda x\pa_x+\omega\pa_y$ means that
$\frac{d}{dt}\tilde\Gamma^{(1)}(x,y,\lambda,\omega,0)=\lambda$. Taylor
expanding $\tilde\Gamma^{(1)}$ further in $t$,
we have
\begin{equation}\label{eq:x-form-sc-flow}
\gamma^{(1)}_{x,y,\lambda,\omega}(t)=x(1+\lambda t+\alpha(x,y,\lambda,\omega)
t^2+t^3\Gamma^{(1)}(x,y,\lambda,\omega,t)),
\end{equation}
with $\Gamma^{(1)}$ smooth. Now, the sublevel sets $\{x<x_0\}$ of $x$ (neighborhoods
of infinity) are assumed to be geodesically concave, so
bicharacteristics tangent to the level set $x=x_0$, i.e.\ with
$\lambda=0$, satisfy
$\frac{d^2}{dt^2}\gamma^{(1)}_{x,y,0,\omega}(t)<0$, so
$\alpha<0$ when $\lambda=0$, cf.\ the discussion after
\eqref{eq:Hessian-x} for seeing that for $x$ sufficiently small,
$\frac{d^2}{dt^2}\gamma^{(1)}_{x,y,0,\omega}(t)$ is bounded from above
by $x$ times a negative constant.
Since the metric is Riemannian, it gives rise to an identification
between $\Tsc^*M$ and $\Tsc M$; the pushforward of $H_g$ to the base
manifold is simply the sc-covector at which we are doing the pushforward so identified with a sc-tangent
vector, and thus the pushforward of $\scH_g=x^{-1}H_g$ to the base is the
sc-covector identified as a b-covector via division by $x$. In particular, for a warped
product-type sc-metric we have $g=\tau^2+H(y,\mu)$, and then the
pushforward of $\scH_g$ from $(x,y,\tau,\mu)$ is
$\tau(x\pa_x)+H(y)(\mu,.)$, i.e.\ $\lambda=\tau$, and $\omega$ is
the standard identification of $\mu$ with a tangent vector on the
cross section $\pa X$.

Notice that \eqref{eq:x-form-sc-flow} implies that
there are $T_0>0,C>0$ such that for $|t|<T_0$,
$$
\gamma^{(1)}_{x,y,\lambda,\omega}(t)\leq x(1+\lambda t-Ct^2)=x\Big(1-C\Big(t-\frac{\lambda}{2C}\Big)^2+\frac{\lambda^2}{4C}\Big).
$$
So if $|\lambda|\leq C_0 \ep$, then
$\gamma^{(1)}_{x,y,\lambda,\omega}(t)\leq x(1+C_1 \ep^2)$, hence
$$
\gamma^{(1)}_{x,y,\lambda,\omega}(t)^{-2}\geq x^{-2}(1-C_2 \ep^2),
$$
hence
$$
x^{-2}-\gamma^{(1)}_{x,y,\lambda,\omega}(t)^{-2}\leq C_2x^{-2}\ep^2.
$$
In particular, if $|\lambda|<C_3 x$, then this quantity is bounded
above by a constant. The weight we use is the exponential of this
(once in the conjugation), and
thus it is bounded. Indeed, this relationship between the dynamics and
the weight motivates the choice of the latter: any larger weight would
mean the operator is not in our pseudodifferential algebra and
any smaller weight would mean that it is irrelevant, and we would not
have an elliptic operator. In particular, this explains that if we
instead had $|\lambda|<C_3x^p$, our weight would be of the form
$e^{-1/x^{2p}}$, and similar results would apply.

A similar argument also implies that for $|t|>C_4|\lambda|$, but
$|t|<T_0$, we have
\begin{equation}\label{eq:exponent-Gaussian-bound}
x^{-2}-\gamma^{(1)}_{x,y,\lambda,\omega}(t)^{-2}\leq -C_5x^{-2}t^2,
\end{equation}
so in particular if $t$ is bounded away from $0$ this is bounded from
above by a negative multiple of $x^{-2}$.

Due to the convexity of the level sets of $x$,
$\gamma^{(1)}_{x,y,\lambda,\omega}(t)$ can have only one local
maximum, which is necessarily near $t=0$, within $C_4|\lambda|$ of it,
so if $|\lambda|<C_3 x$ then within $C_4'x$ of it. Correspondingly, in
fact for $t$ bounded \eqref{eq:exponent-Gaussian-bound} automatically
holds in $|t|>C_4|\lambda|$ (even for $t$ with not necessarily $|t|<T_0$).

We will also need a no conjugate points requirement. We will say that
the metric satisfies the no-conjugate points assumptions if for all
$(x,y)$ with $x<x_0$ the smooth map
$$
(t,\lambda,\omega)\mapsto (x^{-1}\gamma^{(1)}_{x,y,\lambda,\omega}(t),
\gamma^{(2)}_{x,y,\lambda,\omega}(t))=
(\tilde\Gamma^{(1)}_{x,y,\lambda,\omega}(t), \gamma^{(2)}_{x,y,\lambda,\omega}(t))
$$
has a full rank differential. For $x\neq 0$, this is directly
equivalent to the usual statement (since the factor $x^{-1}$ is
irrelevant), but this is the uniform version we need. Note that taking
into account the smooth dependence of the flow on the parameters,
differentiating
$\frac{d\gamma^{(1)}(t)}{dt}=\gamma^{(1)}(t)(\tau+\gamma^{(1)}(t)f_1)$
with respect to $x$ and evaluating at $x=0$ yields
$$
\frac{d}{dt}\pa_1\gamma^{(1)}(t)=\tau\pa_1\gamma^{(1)}(t),
$$
so $\pa_1\gamma^{(1)}(t)$ satisfies a first order homogeneous linear
ODE which only depends on the asymptotic conic metric
$g_\infty$. Since the non-degeneracy condition for $x$
sufficiently small follows from that for $x=0$, we conclude that it
suffices to check the non-degeneracy condition for $g_\infty$, which
we do below after a further reparameterization.

Thus,
\begin{equation*}\begin{aligned}
\gamma_{x,y,\lambda,\omega}(t)&=(\gamma^{(1)}_{x,y,\lambda,\omega}(t)
, \gamma^{(2)}_{x,y,\lambda,\omega}(t))\\
&=(x+x(\lambda t+\alpha
t^2+t^3\Gamma^{(1)}(x,y,\lambda,\omega,t)),y+\omega t+t^2\Gamma^{(2)}(x,y,\lambda,\omega,t))
\end{aligned}\end{equation*}
with $\Gamma^{(1)},\Gamma^{(2)}$
smooth functions of $x,y,\lambda,\omega,t$.

 In order to obtain
uniformity as $|t|\to\infty$, it is very useful to change the parameterization again,
keeping in mind the global nature of the Hamilton flow. The key point
is that as the dual metric function at $x=0$ is $g|_{x=0}=\tau^2+h(y,\mu)$, as a b-vector field,
$$
\scH_g|_{x=0}=2\tau(x\pa_x+\mu\pa_{\mu})-2|\mu|^2\pa_{\tau}+H_h,
$$
as discussed earlier. Correspondingly, the points
$x=0$, $\mu=0$ are `radial points', where this vector field,
considered as a b-vector field, is a
multiple of $x\pa_x$. It is thus natural to blow these up in
$\Tsc^*M$. Given our choice of cutoff, the integral curves of concern
approach $x=0$, $\mu=0$ almost tangent to the boundary $x=0$; in
this region $\rho=\frac{x}{|\mu|}$, $|\mu|$,
$\widehat\mu=\frac{\mu}{|\mu|}$, together with $\tau,y$, are
coordinates on the blown up space; $|\mu|$ defines the front face,
$\rho$ defines the lift of the original boundary. Then $\scH_g=|\mu|
V$, $V$ a vector field tangent to the lift of the original boundary,
$\rho=0$, but transversal to the front face. Correspondingly, the
integral curves of $V$, $\hat\gamma=\hat\gamma_{x,y,\lambda,\omega}(r)$,
with $\frac{dr}{dt}=|\mu(\gamma(t))|$, intersect the front face in
finite time, which enables standard flow arguments. Concretely, taking into account that $|\mu|$ is close to
$1$ at the initial point, so $\rho$ and $x$ can be used
interchangeably there,
$$
\rho=x F(x,y,\lambda,\omega,r),
$$
with $F$
smooth and positive, and $|\mu|$ also a smooth function of
$x,y,\lambda,\omega,r$, hence their product, $x(\hat\gamma(r))$, is
also such a multiple of $x$, with the multiple going to $0$ at the
front face, and
the same is true for $y(\hat\gamma(r))$. Moreover,
$$
|\mu(\hat\gamma(r))|=x(\hat\gamma(r))/\rho(\hat\gamma(r))=x^{-1}
x(\hat\gamma(r)) F(x,y,\lambda,\omega,r)^{-1}.
$$
Note also that if
$R(x,y,\lambda,\omega)$ is the value of $r$ the front face is reached
by the integral curve, $|\mu(\hat\gamma(r))|\sim |r-R|$ (the right
hand side is bounded
above and below by positive multiples of $|r-R|$), so
$x(\hat\gamma(r))$ is a smooth non-degenerate multiple of $x(R-r)$. The
exponential weight we computed above in \eqref{eq:exponent-Gaussian-bound} is
\begin{equation}\begin{aligned}\label{eq:exp-weight-radial}
    &x^{-2}-\hat\gamma^{(1)}_{x,y,\lambda,\omega}(r)^{-2}\\
    &=(\hat\gamma^{(1)}_{x,y,\lambda,\omega}(r)^2-x^2)x^{-2}(\hat\gamma^{(1)}_{x,y,\lambda,\omega}(r))^{-2}\\
    &=-(\hat\gamma^{(1)}_{x,y,\lambda,\omega}(r))^{-2}W(x,y,\lambda,\omega,r),
  \end{aligned}\end{equation}
with $W\to 1$ as $r\to R$.

We also need a no-conjugate point
assumption which analogously to the finite $t$ case means the non-degeneracy of the smooth function
$(x^{-1}\hat\gamma^{(1)},\hat\gamma^{(2)})$ of $(r,\lambda,\omega)$;
here
in fact polynomial in $\hat\gamma^{(1)}$ degeneracy of the
derivative is
acceptable due to the exponential decay of the weight. Notice that for
$r$ away from the endpoints of the interval, this is equivalent to
$(x^{-1}\gamma^{(1)},\gamma^{(2)})$ being non-degenerate as a
function of $(t,\lambda,\omega)$ as
$\frac{dr}{dt}=|\mu(\gamma(t))|\neq 0$ there, i.e.\ it reduces to
the previous discussion. As discussed above for bounded $t$, using
$\pa_1\hat\gamma=x^{-1}\hat\gamma$, the non-degeneracy for small $x$ follows from
non-degeneracy at $x=0$, which in turn follows from the corresponding
property for the asymptotic metric $g_\infty$.
Hence, evaluating at $x=0$ and using the explicit $g_\infty$-flow from
\eqref{eq:conic-bichar}, we have
$$
(x^{-1}\hat\gamma^{(1)},\hat\gamma^{(2)})=\Big(\frac{\sin(r+r_0)}{\sin r_0},\exp(rH_{h/2})(y_0,\hat\mu_0)\Big)
$$
with $\omega=h^{-1}(\hat\mu_0)$ and $\lambda=\cot r_0$, so the
non-degeneracy is equivalent to $(x^{-1}\hat\gamma^{(1)},\hat\gamma^{(2)})$ being non-degenerate as a
function of $(r,r_0,\hat\mu_0)$, which in turn is immediately seen as
being equivalent to the absence of conjugate points under the boundary
metric $h$ within distance $\pi/2$.

\subsection{Invertibility of the geodesic X-ray transform on a collar
  neighborhood of infinity}\label{sec:X-ray-inv}
After these geometric preliminaries we work out the form of the
Schwartz kernel of $A_h=e^{-\Phi/h}L\tilde\chi Ie^{\Phi/h}$.
Relative to the density $|dz'|=|dx'\,dy'|$, this is, with
$z=(x,y)$, $z'=(x',y')$,
\begin{equation}\begin{aligned}\label{eq:Ah-SK-delta}
K_{A_h}(x,y,x',y')&=\int
e^{-\Phi(x)/h}e^{\Phi(x(\gamma_{x,y,\lambda,\omega}(t)))/h}\tilde\chi(x,y,\lambda/(h^{1/2}x),\omega)
\\
&\qquad\qquad\qquad\delta(z'-\gamma_{z,\lambda,\omega}(t))
\gamma^{(1)}_{z,\lambda,\omega}(t)^{-1}\,dt\,|d\sigma|,
\end{aligned}\end{equation}
since that of $I$ is
$$
K_I(z,\lambda,\omega,z')=\int\delta(z'-\gamma_{z,\lambda,\omega}(t))
\gamma^{(1)}_{z,\lambda,\omega}(t)^{-1}\,dt,
$$
with $\gamma^{(1)}_{z,\lambda,\omega}(t)^{-1}$ corresponding to the
Jacobian factor in \eqref{eq:sc-reparameterized-Xray}.
Hence,
\begin{equation}\begin{aligned}\label{eq:Ah-SK}
    &K_{A_h}(x,y,x',y')\\
    &=(2\pi)^{-n}h^{-n/2-1/2}\int
e^{-\Phi(x)/h}e^{\Phi(x(\gamma_{x,y,\lambda,\omega}(t)))/h}\tilde\chi(x,y,\lambda/(h^{1/2}x),\omega)\\
&\qquad\qquad\qquad e^{-i\xi'(x'-\gamma^{(1)}_{z,\lambda,\omega}(t))/h}e^{-i\eta'\cdot(y'-\gamma^{(2)}_{z,\lambda,\omega}(t))/h^{1/2}}\gamma^{(1)}_{z,\lambda,\omega}(t)^{-1}\,dt\,|d\sigma|\,|d\xi'|\,|d\eta'|,
\end{aligned}\end{equation}
where we wrote the delta distribution as a semiclassical foliation
Fourier transform of the constant function $(2\pi)^{-n}$, with
$$
e^{-i\xi'(x'-\gamma^{(1)}_{z,\lambda,\omega}(t))/h}e^{-i\eta'\cdot(y'-\gamma^{(2)}_{z,\lambda,\omega}(t))/h^{1/2}}
$$
being the kernel of the Fourier transform. The integrand of the
$dt\,|d\sigma|$ integral can be considered as
a semiclassical foliation Fourier transform $(\xi',\eta')\to(x',y')$
of
\begin{equation}\begin{aligned}\label{eq:Ah-SK-IFT}
    (2\pi)^{-n}
    e^{-\Phi(x)/h}e^{\Phi(x(\gamma_{x,y,\lambda,\omega}(t)))/h}&\tilde\chi(x,y,\lambda/(h^{1/2}x),\omega)
    \gamma^{(1)}_{z,\lambda,\omega}(t)^{-1}\\
    &\qquad\qquad 
    e^{i\xi'\gamma^{(1)}_{z,\lambda,\omega}(t)/h}e^{i\eta'\cdot\gamma^{(1)}_{z,\lambda,\omega}(t)/h^{1/2}},
  \end{aligned}\end{equation}
with the factors on the first line independent of the Fourier
transform variables $(\xi',\eta')$. For the purposes below, it is
useful to have the Schwartz kernel relative to the density
$\frac{|dx'\,dy'|}{(x')^{n+2}}$, cf.\ our definition of the 1-cusp algebra. In view of the delta distribution in
\eqref{eq:Ah-SK-delta}, this can be achieved by adding a factor
$(\gamma^{(1)}_{z,\lambda,\omega}(t))^{n+2}$ to \eqref{eq:Ah-SK}, and
thus to \eqref{eq:Ah-SK-IFT}.

In order to proceed, we recall from \eqref{eq:1c-quantize} that
1-cusp operators are given by the oscillatory integral
\begin{equation*}\begin{aligned}
    &A_h u(x,y)=Au(x,y,h)\\
    &=(2\pi)^{-n} h^{-n/2-1/2}\int e^{i\Big(\frac{x-x'}{x^3}\frac{\widetilde{\xiocu}}{h}+\frac{y-y'}{x}\frac{\widetilde{\etaocu}}{h^{1/2}}\Big)}a_h(x,y,\widetilde{\xiocu},\widetilde{\etaocu})\,u(x',y')\,\frac{dx'\,dy'}{(x')^{n+2}}\,d\widetilde{\xiocu}\,d\widetilde{\etaocu},
  \end{aligned}\end{equation*}
where $a$ is a standard (conormal) symbol. Thus, the Schwartz kernel is
$$
K_{A_h}(x,y,x',y')=(2\pi)^{-n} h^{-n/2-1/2}\int e^{i\Big(\frac{x-x'}{x^3}\frac{\widetilde{\xiocu}}{h}+\frac{y-y'}{x}\frac{\widetilde{\etaocu}}{h^{1/2}}\Big)}a_h(x,y,\widetilde{\xiocu},\widetilde{\etaocu})\,d\widetilde{\xiocu}\,d\widetilde{\etaocu},
$$
relative to the density $\frac{dx'\,dy'}{(x')^{n+2}}$.
, i.e.\ they are
$(2\pi)^{-n}x^{n+2}$ (with the second factor due to the Jacobian in
scaling the Fourier transform) times the semiclassical foliation Fourier transform in
$(x^{-3}\widetilde{\xiocu},x^{-1}\widetilde{\etaocu})$ of
$$
(x,y,\widetilde{\xiocu},\widetilde{\etaocu})\mapsto e^{i(x^{-2}\widetilde{\xiocu}/h+x^{-1}y\cdot\widetilde{\etaocu}/h^{1/2})} a(x,y,\widetilde{\xiocu},\widetilde{\etaocu}).
$$
Inverting this Fourier transform, evaluating at $(x^{-3}\xiocu,x^{-1}\etaocu)$,
\begin{equation*}\begin{aligned}
    &a_h(x,y,\xiocu,\etaocu)\\
    &=(2\pi)^nx^{-n-2}
    e^{i(-x^{-2}\xiocu/h-x^{-1}y\cdot\etaocu/h^{1/2})}\\
    &\qquad\qquad\qquad(\cF^{-1}_{h,\cF})_{(x',y')\to (x^{-3}\xiocu,x^{-1}\etaocu)}K_{A_h}(x,y,x',y').
\end{aligned}\end{equation*}
Proceeding from \eqref{eq:Ah-SK}, taking into account the Fourier
transform statement following it, we obtain
\begin{equation}\begin{aligned}\label{eq:semicl-ocu-full-symbol}
    &a_h(x,y,\xiocu,\etaocu)\\
    &=x^{-n-2}e^{-ix^{-3}\xiocu
      x/h}e^{-ix^{-1}\etaocu\cdot y/h^{1/2}}\int
e^{-\Phi(x)/h}e^{\Phi(x(\gamma_{x,y,\lambda,\omega}(t)))/h}\tilde\chi(x,y,\lambda/(h^{1/2}x),\omega)\\
&\qquad\qquad\qquad
e^{ix^{-3}\xiocu\gamma^{(1)}_{z,\lambda,\omega}(t)/h}e^{ix^{-1}\etaocu\cdot\gamma^{(2)}_{z,\lambda,\omega}(t)/h^{1/2}}\gamma^{(1)}_{z,\lambda,\omega}(t)^{n+1}\,dt\,|d\sigma|\\
&=\int
e^{-\Phi(x)/h}e^{\Phi(x(\gamma_{x,y,\lambda,\omega}(t)))/h}\tilde\chi(x,y,\lambda/(h^{1/2}x),\omega)\\
&\qquad\qquad\qquad e^{ix^{-3}\xiocu(\gamma^{(1)}_{z,\lambda,\omega}(t)-x)/h}e^{ix^{-1}\etaocu\cdot(\gamma^{(2)}_{z,\lambda,\omega}(t)-y)/h^{1/2}}x^{-n-2}\gamma^{(1)}_{z,\lambda,\omega}(t)^{n+1}\,dt\,|d\sigma|.
\end{aligned}\end{equation}
Note that this corresponds to \cite[Equation~(3.8)]{Vasy:Semiclassical-X-ray},  taking into account the factor in
\eqref{eq:sc-reparameterized-Xray}, and that we write covectors as $\xiocu\frac{dx}{x^3}+\etaocu\frac{dy}{x}$, and thus
$\xi=x^{-3}\xiocu$, $\eta=x^{-1}\etaocu$ in
\cite[Equation~(3.8)]{Vasy:Semiclassical-X-ray}.
Recall that $\Phi(x)=-\frac{1}{2x^2}$ here, and
$\gamma^{(1)}_{z,\lambda,\omega}(t)^{-1}$ is bounded by
$Cx^{-1}(1+|t|)$. Thus, it remains to show that the right hand side of
\eqref{eq:semicl-ocu-full-symbol} is
$h$ times
a symbol of order $-1,-1$, i.e.\ $hx$ times a symbol of order $-1,0$.

We remark that \eqref{eq:semicl-ocu-full-symbol} uses local
coordinates. In general, for the Schwartz kernel $K_A(z,z')$ of our semiclassical operators we should be
considering both the possibilities that $z$ and $z'$ are in the same
chart, and also that they are away from each other, in different
charts. In the latter case $|t|$ is necessarily bounded from below by a
positive constant, and the argument below, discussed in the notation
of the same chart, applies directly and shows that the Schwartz
kernel is Schwartz and is $O(h^\infty)$. Indeed, in the argument given
when discussing that region in $t$ in the oscillatory integral the unprimed
variables can be regarded as fixed, so for many purposes there is not even a need to
consider a coordinate chart explicitly, and in any case the device
described for the cusp pseudodifferential operators of taking two
disjoint open sets $O,U$ and identifying them with different open sets
of $\RR^n$, now in the parabolic sense, would be applicable. In the
direct treatment one can describe the Schwartz kernel directly (as
opposed to through the symbol, which requires a Fourier transform)
which is residual in these regions in terms of \eqref{eq:Ah-SK}; our
computations then directly show that prior to the $|d\xi'|\,|d\eta'|$
integrals one already has rapid decay and smoothness in all variables
(both $x,h$ and $\xi',\eta'$).

We break up our analysis into four regions by the use of a partition
of unity (which we suppress in notation): $|t|<C_1h^{1/2}x$,
$C_2h^{1/2}x<|t|<T_0$ (with $T_0$ corresponding to both flow and
coordinate considerations, so it is sufficiently small and positive),
$0<C_3<|t|<C_4$ ($C_3,C_4$ arbitrary positive) and $|t|$ near infinity, though
the third and fourth regions can be combined.

We first analyze the pseudodifferential, i.e.\ near diagonal,
behavior of the Schwartz kernel. For this, we may take $z,z'$ in the
same chart and $|t|$ bounded by a constant $T_0>0$. In order to proceed,
we change the variables of integration to $\hat t=t/(h^{1/2} x)$
and $\hat\lambda=\lambda/(h^{1/2}x)$, so the $\hat\lambda$ integration is over
a fixed interval. The phase is
\begin{equation}\begin{aligned}\label{eq:semicl-ocu-phase}
    &x^{-3}\xiocu\cdot(\gamma^{(1)}_{z,\lambda,\omega}(t)-x)/h+x^{-1}\etaocu\cdot(\gamma^{(2)}_{z,\lambda,\omega}(t)-y)/h^{1/2}\\
&=\xiocu(\hat\lambda\hat t+\alpha (x,y,xh^{1/2}\hat\lambda,\omega)\hat t^2+xh^{1/2}\hat
t^3\Gamma^{(1)}(x,y,xh^{1/2}\hat\lambda,\omega,xh^{1/2}\hat t))\\
&\qquad\qquad+\etaocu\cdot(\omega\hat t+xh^{1/2}\hat t^2\Gamma^{(2)}(x,y,xh^{1/2}\hat\lambda,\omega,xh^{1/2}\hat t)),
\end{aligned}\end{equation}
while the exponential damping factor (which we regard as a Schwartz
function, part of the amplitude, when one regards $\hat t$ as a
variable on $\RR$) is (recall that $\alpha<0$!)
\begin{equation}\begin{aligned}\label{eq:semicl-ocu-exp-damping}
    &1/(2hx^2)-1/(2h\gamma^{(1)}_{x,y,\lambda,\omega}(t)^2)\\
    &=\frac{1}{2}h^{-1}(\gamma^{(1)}_{x,y,\lambda,\omega}(t)^2-x^2)x^{-2}(\gamma^{(1)}_{x,y,\lambda,\omega}(t))^{-2}\\
&=
\frac{1}{2} h^{-1}x(\lambda t+\alpha (x,y,xh^{1/2}\hat\lambda,\omega)
t^2+t^3\Gamma^{(1)}(x,y,\lambda,\omega,t))\\
&\qquad\qquad\qquad x(2+\lambda t+\alpha (x,y,xh^{1/2}\hat\lambda,\omega)
t^2+t^3\Gamma^{(1)}(x,y,\lambda,\omega,t))\\
&\qquad\qquad\qquad x^{-2}x^{-2}(1+\lambda t+\alpha (x,y,xh^{1/2}\hat\lambda,\omega)
t^2+t^3\Gamma^{(1)}(x,y,\lambda,\omega,t))^{-2}\\
&=\hat\lambda\hat t+\alpha (x,y,xh^{1/2}\hat\lambda,\omega)\hat
t^2+\hat t^3 xh^{1/2}\hat\Gamma^{(1)}(x,y,xh^{1/2}\hat\lambda,\omega,xh^{1/2}\hat t),
\end{aligned}\end{equation}
with $\hat\Gamma^{(1)}$ a smooth function. Thus, as we explain below
in more detail, for $\xiocu,\etaocu$ in a
bounded region we conclude that $a_h$ is a $\CI$
function of all variables, including $h^{1/2}$. Furthermore, we observe that with $(\xiocu,\etaocu)$ in place
of $(\xi,\eta)$, and in the new integration variables $\hat t$ and
$\hat\lambda$, \eqref{eq:semicl-ocu-full-symbol} has the same form as
\cite[Equation~(3.8)]{Vasy:Semiclassical-X-ray}, so identical stationary phase arguments
are applicable.

\begin{rmk}\label{rmk:1cusp-phase-exp-p}
If we used the scaling $\hat\lambda=\lambda/(h^{1/2}x^p)$ in the definition of
$\tilde\chi$, and replaced $x$ by $x^p$ in the definition of $\Phi$,
with more precisely $\Phi=-\frac{1}{2px^{2p}}$, we
would obtain essentially the same rescaled result for the phase and
the exponential weight. Namely, write
$\widetilde{\xiocu},\widetilde{\etaocu}$ for the (slightly modified) $x^p$-based 1-cusp dual
variables, i.e.\ write covectors as
$$
\widetilde{\xiocu} \frac{dx^p}{x^{3p}}+\widetilde{\etaocu} \frac{dy}{x^{p}}=p \widetilde{\xiocu} \frac{dx}{x^{2p+1}}+\widetilde{\etaocu} \frac{dy}{x^{p}}.
$$
Then the right hand side of the phase,
\eqref{eq:semicl-ocu-phase}, becomes,
and $\hat t=t/(h^{1/2}x^p)$,
\begin{equation*}\begin{aligned}
&p\widetilde{\xiocu}(\hat\lambda\hat t+\alpha (x,y,x^ph^{1/2}\hat\lambda,\omega)\hat t^2+x^ph^{1/2}\hat
t^3\Gamma^{(1)}(x,y,x^ph^{1/2}\hat\lambda,\omega,x^ph^{1/2}\hat t))\\
&\qquad\qquad+\widetilde{\etaocu}\cdot(\omega\hat t+x^ph^{1/2}\hat t^2\Gamma^{(2)}(x,y,x^ph^{1/2}\hat\lambda,\omega,x^ph^{1/2}\hat t)),
\end{aligned}\end{equation*}
and similarly the analogue of the right hand side of
\eqref{eq:semicl-ocu-exp-damping} becomes
\begin{equation*}\begin{aligned}
    \hat\lambda\hat t+\alpha (x,y,x^ph^{1/2}\hat\lambda,\omega)\hat
t^2+\hat t^3 x^ph^{1/2}\hat\Gamma^{(1)}(x,y,x^ph^{1/2}\hat\lambda,\omega,x^ph^{1/2}\hat t).
\end{aligned}\end{equation*}
These allow all the arguments below to proceed without any significant
change; even the ellipticity computation is unaffected apart from
scaling $\widetilde{\xiocu}$ by an irrelevant factor of $p$.
  \end{rmk}

Concretely, upon the rescaling $t$ and $\lambda$ to $\hat t$ and
$\hat\lambda$, which introduces a factor of $hx^2$ from the Jacobian, the integrand of \eqref{eq:semicl-ocu-full-symbol} is
$xh$ times a
smooth function of all variables, integrated in a compact region
except in $\hat t$. However, the Gaussian decay, in view of
\eqref{eq:exponent-Gaussian-bound} and \eqref{eq:semicl-ocu-exp-damping}, of the exponential
damping factor
$$
e^{-\Phi(x)/h} e^{\Phi(x(\gamma_{z,\lambda,\omega}(t)))/h}
$$
means that this non-compactness is not an issue, and
\eqref{eq:semicl-ocu-full-symbol} itself is $hx$ times a smooth
function of all variables, in accordance to the desired $hx$ times a
symbol of order $-1,0$ conclusion, namely showing the smoothness part,
but not yet the estimates as $|(\xiocu,\etaocu)|\to\infty$.

We now consider the $|(\xiocu,\etaocu)|\to\infty$ behavior. We use the
stationary phase lemma, but with the slight complication that the
$\hat t$ integration interval is non-compact. In order to deal with
this, we divide the integration region into one in which $|\hat t|$
bounded, resp.\ one in which $|\hat t|\geq 1$. In the the former one
can use the standard parameter dependent version of the stationary phase lemma, while in
the latter the phase is non-stationary and one can use a direct
integration by parts argument.

Starting with the former, at
$h^{1/2}=0$, the phase is
$$
\xiocu(\hat\lambda\hat t+\alpha (x,y,0,\omega)\hat t^2)+\etaocu\cdot
\omega\hat t.
$$
Consider first $\xiocu\neq 0$: taking the $\hat\lambda$ derivative shows that
$\hat t=0$ at the critical set, and thus taking the $\hat t$
derivative shows that $\xiocu\hat\lambda+\etaocu\cdot\omega=0$, i.e.\
$\hat\lambda=-\xiocu^{-1}\etaocu\cdot\omega$; this is actually
critical with respect to the full set $(\hat
t,\hat\lambda,\omega)$. Moreover, this set remains critical for
$h^{1/2}$ non-zero due to the $\hat t^2$ vanishing factors in other
terms of the phase. At $h^{1/2}=0$ the Hessian of the phase in $(\hat t,\hat\lambda)$
at the critical set
is
$$
\begin{pmatrix} 2\xiocu \alpha
  (x,y,0,\omega)&\xiocu\\\xiocu&0\end{pmatrix},
$$
which is invertible,
with determinant $-\xiocu^2$,
hence remains so for small $h$. Thus,
regarding $\omega$ as a parameter, the stationary phase lemma applies
and yields that in this region $a_{h}$ is $xh$ times (due to the
Jacobian factor discussed above!) a symbol of order $-1$ (from
the reciprocal of the square root of the Hessian determinant). On the
other hand, when $\etaocu\neq 0$ (for the behavior as
$|(\xiocu,\etaocu)|\to\infty$ we only need to consider when at least
one of $\xiocu$ and $\etaocu$ is non-zero), decompose $\omega$ corresponding to
$\etaocu$ into a parallel and an orthogonal component, writing
$\omega^\parallel=\omega\cdot\widehat\etaocu$,
$\widehat\etaocu=\frac{\etaocu}{|\etaocu|}$, so the phase at $h^{1/2}=0$ becomes 
\begin{equation}\label{eq:phase-etaocu-large}
|\etaocu|\Big(\frac{\xiocu}{|\etaocu|}(\hat\lambda\hat t+\alpha (x,y,0,\omega)\hat t^2)+
\omega^\parallel\hat t\Big).
\end{equation}
First, taking the $\hat\lambda$ derivative shows that either
$\frac{\xiocu}{|\etaocu|}=0$ or $\hat t=0$ at the critical set, and in
the former case taking the $\hat t$ derivative shows that
$\omega^\parallel=0$, while in the latter case the same $\hat t$
derivative shows that (as we already have $\hat t=0$)
$\frac{\xiocu}{|\etaocu|}\hat\lambda+\omega^\parallel=0$, so in view
of the boundedness of $\hat\lambda$ and as we may assume the smallness
of $\frac{\xiocu}{|\etaocu|}$ in view of the already treated case,
$|\omega^\parallel|$ is bounded away from $1$, and thus
$\omega^\parallel$ is a valid coordinate at the critical set, $\hat t=0$,
$\frac{\xiocu}{|\etaocu|}\hat\lambda+\omega^\parallel=0$, which is indeed
critical with respect to the full set $(\hat t,\hat
\lambda,\omega^\parallel,\omega^\perp)$ of parameters. Further, this
remains also true for $h^{1/2}$ non-zero due to the $\hat t^2$
vanishing factors in the other terms of the phase. The Hessian with
respect to $(\hat t,\omega^\parallel)$ is
$$
|\etaocu|\begin{pmatrix}2\frac{\xiocu}{|\etaocu|}\alpha
  (x,y,0,\omega)&1\\1&0\end{pmatrix},
$$
which is again invertible,
  with determinant $-|\etaocu|^2$, and thus remains so for $h^{1/2}$
  small. Thus, regarding $\hat\lambda,\omega^\perp$ as parameters, the
  stationary phase lemma applies and yields that in this region as
  well $a_{h}$ is $xh$ times a symbol of order $-1$. This completes
  the proof of the symbolic behavior of the contribution of the
  oscillatory intergral \eqref{eq:semicl-ocu-full-symbol} from $\hat
  t$ bounded.

In hindsight, as this will be useful for the symbolic computation, we
can rewrite the phase by regarding $\theta=(\hat\lambda,\omega)$ and
$(\xiocu,\etaocu)$ jointly, writing the latter as $|(\xiocu,\etaocu)|(\widehat\etaocu,\widehat\xiocu)$:
$$
|(\xiocu,\etaocu)|(\widehat\xiocu(\hat\lambda\hat t+\alpha (x,y,0,\omega)\hat t^2)+\widehat\etaocu\cdot
\omega\hat t.
$$
Decomposing $\theta$ into parallel and orthogonal components relative
to $(\widehat\etaocu,\widehat\xiocu)$, so
$\theta^\parallel=(\widehat\etaocu,\widehat\xiocu)\cdot\theta$, the
phase is the large parameter $|(\xiocu,\etaocu)|$ times
\begin{equation}\label{eq:phase-theta}
\theta^\parallel\hat t+\widehat\xiocu\alpha (x,y,0,\omega)\hat t^2.
\end{equation}
As we already know that at the critical set $\hat t=0$, we deduce that
it is given by $\theta^\parallel=0$, and that the Hessian of the phase
there with respect to $(\hat t,\theta^\parallel)$ is
$$
|(\xiocu,\etaocu)|\begin{pmatrix}2 \widehat\xiocu\alpha
  (x,y,0,\omega)&1\\1&0\end{pmatrix}.
$$

We next analyze the $|t|\leq T_0$ small, $|\hat t|\geq 1$ region. Here
we use a direct integration by parts argument, utilizing that if the
derivative of the phase with respect to one of the integration
variables $(\hat t,\hat\lambda,\omega)$ is bounded below by a positive
multiple of $|(\xiocu,\etaocu)| |\hat t|^{-k}$ for some $k$,
integration by parts in this variable, taking into account the
Gaussian exponential damping factor bounded by $e^{-\ep
  t^2/(hx^2)}=e^{-\ep\hat t^2}$ in $|t|>Ch^{1/2}x$ by \eqref{eq:exponent-Gaussian-bound}, gives rapid decay of the
integral with respect to the large parameter
$|(\xiocu,\etaocu)|$. Hence, it remains to check that in all regions
one has such a lower bound for some derivative; again it suffices to
check this at $h^{1/2}=0$. If $\xiocu\neq 0$, then
the $\hat\lambda$ derivative of the phase is $\xiocu\hat t$, giving
the desired statement. If $\etaocu\neq 0$, the form
\eqref{eq:phase-etaocu-large} of the phase shows that first of all the
$\hat\lambda$ derivative has such a lower bound as soon as
$\frac{|\xiocu|}{|\etaocu|}$ is bounded from below by $|\hat t|^{-2}$. Then
as long as
$\omega^\parallel$ is a valid coordinate, the derivative with respect to $\omega^\parallel$ is
$|\etaocu|\hat t$, giving the desired lower bound. The remaining case, when
$\omega^\parallel$ is not a valid coordinate, i.e.\ when
$|\omega^\parallel|$ is close to $1$, and
$\frac{|\xiocu|}{|\etaocu|}\leq |\hat t|^{-2}$. In this case the $\hat t$
derivative becomes
$|\etaocu|(\frac{\xiocu}{|\etaocu|}(\hat\lambda+2\alpha\hat
t)+\omega^\parallel)$, which is now bounded away from $0$, completing
the proof of the direct integration by parts argument in the region $|t|\leq T_0$ small, $|\hat t|\geq 1$.

Consider now the region where $t$ is bounded away from $0$, but is
bounded; in this case the exponential weight is bounded by
$e^{-\ep/(x^2h)}$, cf.\ \eqref{eq:exponent-Gaussian-bound}, thus is rapidly decaying. Recall that the phase is
\begin{equation*}\begin{aligned}
    &x^{-3}\xiocu\cdot(\gamma^{(1)}_{z,\lambda,\omega}(t)-x)/h+x^{-1}\etaocu\cdot(\gamma^{(2)}_{z,\lambda,\omega}(t)-y)/h^{1/2}\\
    &\qquad=x^{-1}h^{-1/2}\Big((x^{-1}\xiocu/h^{1/2})(x^{-1}\gamma^{(1)}_{z,\lambda,\omega}(t)-1)+\etaocu\cdot(\gamma^{(2)}_{z,\lambda,\omega}(t)-y)\Big),
\end{aligned}\end{equation*}
and $\pa_t(x^{-1}\gamma^{(1)})$ is non-zero (bounded away from $0$) in this region by the convexity
properties of the foliation. Thus, for $t$ away from $0$ there is $C_0>0$ such that if $|\xiocu| x^{-1}/h^{1/2}>C_0|\etaocu|$
    then the phase is non-stationary with respect to $t$, hence the
    integral is rapidly decaying in
    $|(\xiocu,\etaocu)|/(h^{1/2}x)$. On the other hand,
under the no conjugate points assumption, in the precise sense
described in Section~\ref{sec:geodesic-structure}, letting $\widetilde\xiocu=x^{-1}\xiocu/h^{1/2}$, if $|\widetilde\xiocu|=|\xiocu|
    x^{-1}/h^{1/2}<2C_0|\etaocu|$, one has the standard
    no-conjugate points argument available as the phase is a standard
    homogeneous degree 1 phase in $(\widetilde\xiocu,\etaocu)$ times
    $x^{-1}h^{-1/2}$. Here the no-conjugate points argument uses the
    non-degenerateness (full rank of the Jacobian) of the smooth function
    $(x^{-1}\gamma^{(1)},\gamma^{(2)})$ as a function of
    $(t,\lambda,\omega)$, $t\neq 0$, as discussed in
    Section~\ref{sec:geodesic-structure}.

Thus, it remains to consider $|t|\to\infty$.  Following Section~\ref{sec:geodesic-structure},
it is very useful to change the parameterization again,
keeping in mind the global nature of the Hamilton flow, to $r$ from
$t$. Recall that here we write $\hat\gamma=\hat\gamma_{x,y,\lambda,\omega}(r)$,
with $\frac{dr}{dt}=|\etasc(\gamma(t))|$, and the integral curves
intersect the boundary in finite time; this enables standard
integration by parts arguments. As follows from \eqref{eq:exp-weight-radial}, the amplitude is exponentially
decaying in $h^{-1}(\hat\gamma^{(1)}_{x,y,\lambda,\omega}(r))^{-2}$,
thus together with the extra factor
$$
x(\hat\gamma(r))^{-1}|\etasc(\hat\gamma(r))|^{-1}=x(\hat\gamma(r))^{-2}x
F(x,y,\lambda,\omega,r)
$$
it still has this property, suppressing the
endpoint of the bicharacteristic. Now, if $|\xiocu|
x^{-1}/h^{1/2}>C_0|\etaocu|$ with a sufficiently large $C_0$, then the
phase is non-stationary with respect to $r$, giving the desired decay
result; otherwise the no-conjugate points assumption achieves
this. Again, as discussed in Section~\ref{sec:geodesic-structure}, for this the non-degeneracy of the smooth function
$(x^{-1}\hat\gamma^{(1)},\hat\gamma^{(2)})$ of $(r,\lambda,\omega)$ is
used, where in fact polynomial in $\hat\gamma^{(1)}$ degeneracy of the
derivative is
acceptable due to the exponential decay of the weight. In combination
this proves the claimed pseudodifferential property.

Finally, the ellipticity computation is as in
\cite[Equation~(3.12)]{Vasy:Semiclassical-X-ray} with $(\xiocu,\etaocu)$ in place of
$(\xisc, \etasc)$. Concretely, in order to compute the semiclassical
principal symbol from  \eqref{eq:semicl-ocu-full-symbol}, we may
simply let $h^{1/2}=0$ in the rescaled expression, apart from the
overall prefactor, so
\begin{equation*}\begin{aligned}
a_{h}(x,y,\xiocu,\etaocu)=
xh\int &e^{i(\xiocu(\hat\lambda\hat t+\alpha(x,y,0,\omega)\hat
  t^2)+\etaocu\cdot\omega\hat t)} \\
&\qquad e^{\hat\lambda\hat t+\alpha(x,y,0,\omega)\hat t^2}\tilde\chi(z,\hat\lambda,\omega) 
\,d\hat
t\,d\hat\lambda\,d\omega,
\end{aligned}\end{equation*}
up to errors gaining $O(xh^{1/2}\langle\xiocu,\etaocu\rangle^{-1})$ relative
to the leading order $O(xh\langle\xiocu,\etaocu\rangle^{-1})$. In
order to compute the principal symbol of this, i.e.\ the behavior as $|(\xiocu,\etaocu)|\to\infty$, we
recall from \eqref{eq:phase-theta} that it is useful to regard
$\theta=(\hat\lambda,\omega)$ as a joint variable, decomposed relative
to $(\widehat\etaocu,\widehat\xiocu)$, with critical set given by
$\hat t=0$, $\theta^\parallel=0$. Thus, by the stationary phase lemma, the principal symbol of the
semiclassical principal symbol is an elliptic multiple of
$$
\int_{\sphere^{n-2}} \tilde\chi(z,\hat\lambda(\theta^\perp),\omega(\theta^\perp))\,d\theta^{\perp},
$$
which is elliptic for $\chi\geq 0$ with $\chi(0,\cdot)>0$ since the
codimension one planes $\theta^\parallel=0$ and $\hat\lambda=0$
necessarily intersect in a line through the origin, and thus
non-trivially intersect the sphere as $n\geq 2+1=3$. Hence it remains
to compute the semiclassical principal symbol at finite points.

The computation at finite points is more difficult, but it simplifies
greatly if we take $\chi$ to be a Gaussian (which is not technically allowed, but
we will approximate it below). Namely, recalling that
$\alpha<0$, the semiclassical principal symbol is, for $c$ a non-zero constant,
\begin{equation*}\begin{aligned}
xh\int &e^{\alpha\hat t^2(1+i\xiocu)+\hat t(\hat\lambda(1+i\xiocu)+i\etaocu\cdot\omega)}\tilde\chi(z,\hat\lambda,\omega) 
\,d\hat
t\,d\hat\lambda\,d\omega\\
&=xh\int e^{\alpha(1+i\xiocu)(\hat t+\frac{\hat\lambda(1+i\xiocu)+i\etaocu\cdot\omega}{2\alpha(1+i\xiocu)})^2}e^{\frac{(\hat\lambda(1+i\xiocu)+i\etaocu\cdot\omega)^2}{4\alpha(1+i\xiocu)}}\tilde\chi(z,\hat\lambda,\omega) 
\,d\hat
t\,d\hat\lambda\,d\omega\\
&=cxh\int |\alpha|^{-1/2}(1+i\xiocu)^{-1/2}e^{-\frac{(\hat\lambda(1+i\xiocu)+i\etaocu\cdot\omega)^2}{4\alpha(1+i\xiocu)}}\tilde\chi(z,\hat\lambda,\omega) 
\,d\hat\lambda\,d\omega.
\end{aligned}\end{equation*}
Letting $\tilde\chi(z,\hat\lambda,\omega)=e^{\hat\lambda^2/(2\alpha)}$ this can be
rewritten as
\begin{equation*}\begin{aligned}
    &cxh\int
    |\alpha|^{-1/2}(1+i\xiocu)^{-1/2}e^{\frac{\hat\lambda^2}{2\alpha}}
    e^{-\frac{\hat\lambda^2(1+i\xiocu)}{4\alpha}}
    e^{-i\frac{\hat\lambda\etaocu\cdot\omega}{2\alpha}} e^{\frac{(\etaocu\cdot\omega)^2}{4\alpha(1+i\xiocu)}}
    \,d\hat\lambda\,d\omega\\
    &=cxh\int
    |\alpha|^{-1/2}(1+i\xiocu)^{-1/2}
    e^{\frac{\hat\lambda^2(1-i\xiocu)}{4\alpha}}
    e^{-i\frac{\hat\lambda\etaocu\cdot\omega}{2\alpha}} e^{\frac{(\etaocu\cdot\omega)^2}{4\alpha(1+i\xiocu)}}
    \,d\hat\lambda\,d\omega\\
    &=cxh\int
    |\alpha|^{-1/2}(1+i\xiocu)^{-1/2}
    e^{\frac{1-i\xiocu}{4\alpha}(\hat\lambda-i\frac{\etaocu\cdot\omega}{1-i\xiocu})^2}
    e^{\frac{(\etaocu\cdot\omega)^2}{4\alpha(1-i\xiocu)}} e^{\frac{(\etaocu\cdot\omega)^2}{4\alpha(1+i\xiocu)}}
    \,d\hat\lambda\,d\omega\\
    &=c'xh\int
    |\alpha|^{-1}(1+\xiocu^2)^{-1/2}
    e^{\frac{(\etaocu\cdot\omega)^2}{2\alpha(1+\xiocu^2)}}
    \,d\omega
  \end{aligned}\end{equation*}
with $c'$ non-zero, and now the integral is positive since the
integrand is such. Since we need $\chi$ to be compactly supported, we
approximate the Gaussian in the space of Schwartz functions by
compactly supported $\chi$; for suitable approximation the same
positivity property follows. This completes the proof of the
ellipticity, and thus the proof of Theorem~\ref{ispsdo}.

\subsection{Consequences of Theorem~\ref{ispsdo}}\label{sec:ispsdo-conseq}
Having proved Theorem~\ref{ispsdo}, we can apply the results from Section~\ref{sec:semicl-alg} to the modified
normal operator $A$. This means that there is a parametrix $B$ in this
new operator class, where the errors $A \circ B - \Id$ and $B \circ A
- \Id$ are residual operators in the semiclassical foliation 1-cusp
algebra. However, this ellipticity only applied for $x_0 \le \bar
x_0$, and so these residual errors are only residual on the operator
over this domain. This can be used by viewing the operator $A$
acting on functions with support in this region $x_0 \le \bar
x_0$. For functions supported in this collar region we conclude:

\begin{crl}\label{cor:main-conic-op-support}
The modified normal operator $A$, in a region where $x \le \bar x_0$, has a left-parametrix which in the region $x \le \bar x_0$ is in the class $h^{-1}\Psiocuhh^{1, 1}$ 
with error in $h^\infty\Psiocuhh^{-\infty, -\infty}$, and therefore
for any sufficiently small $h$, $A$ is
left invertible on functions supported in this region. 
\end{crl}

\begin{proof}[Proof of Corollary]
Let $O$ be a collar neighborhood of $\pa\overline{M}$ on which $A$ is
elliptic, and let $K=\{x\leq\bar x_0\}\subset O$. Let $\phi$ be a
cutoff function, identically $1$ on $K$, supported in $O$. Let $O'$ be
open
with $\overline {O'}\subset O$ and $\supp\phi\subset O'$.
Then ellipticity gives us that there is an operator $B \in
h^{-1}\Psiocuhh^{1,1}$, such that the errors $E_1 = \Id - A \circ B$
and $E_2 = \Id - B \circ A$, while globally only satisfy that $E_1,
E_2 \in \Psiocuhh^{0,0}$, but locally on $O'$ these errors are
residual, and thus $\phi E_i \phi \in h^\infty\Psiocuhh^{-\infty,
  -\infty}$, $i=1,2$. Now, $\phi BA\phi=\phi^2+\phi E_2\phi$, and for
$v$ supported in $K$, $\phi v=v$, so
$$
\phi BA v=v+\phi E_2\phi v=(\Id+\phi E_2\phi)v.
$$
Now, $\phi E_2\phi$ is $O(h^\infty)$ as a bounded operator on any
weighted Sobolev space, so for $h$ sufficiently small $\Id+\phi
E_2\phi$ is invertible, and hence
$$
v=(\Id+\phi E_2\phi)^{-1}\phi BAv.
$$
This completes the proof of the stated left invertibility.
\end{proof}

In view if the definition of $A$, taking a sufficiently small $h$, this immediately implies:

\begin{crl}\label{cor:main-conic-support}
  There is a collar neighborhood of the boundary such that
the (local) geodesic X-ray $I$ is injective on sufficiently fast Gaussian
decaying functions supported in this neighborhood.
\end{crl}

\subsection{Artificial boundary}\label{sec:combined-sc-1c}
In order to prove the main result, Theorem~\ref{thm:main-conic}, we simply need to add an artificial
boundary, $x=\bar x_0$. We then work on the domain
$\overline{\Omega}=\{x\leq\bar x_0\}$, which has two disjoint boundary
hypersurfaces, $x=0$ and $x=\bar x_0$. We work with a foliation
semiclassical algebra corresponding to the level sets of $x$ such that
in addition at $x=0$ the algebra is 1-cusp, while at $x=\bar x_0$ it
is scattering. Since the two boundary hypersurfaces are disjoint, this
joint algebra can simply defined by localization. Indeed,  we have
already discussed the semiclassical foliation 1-cusp algebra
$\Psiocuhh$, which gives the localized behavior near $x=0$ (or more
strongly away from $x=\bar x_0$). In addition in
\cite{Vasy:Semiclassical-X-ray} the semiclassical foliation algebra
has been defined; this is the model near $x=\bar x_0$ (and more
strongly away from $x=0$). In both algebras if $\phi,\psi$ are $\CI$
(on the compact underlying manifold) with disjoint support, the Schwartz kernels
$\phi A\psi$, where $A$ is an element of the algebra, are
$\CI$ and rapidly decreasing both in $h$ and at the boundary. Thus,
one can define the joint algebra, $\Psiscocuhh$ by:

\begin{dfn}
The space $\Psiscocuhh^{m,l_1,l_2}$ consists of operators $A$ on
$\dCI(\overline{\Omega})$ such that
\begin{enumerate}
\item
  If $\phi,\psi\in \CI(\overline{\Omega})$ with support disjoint from
  $x=0$, then $\phi A\psi\in\Psischh^{m,l_1}$.
\item
  If $\phi,\psi\in \CI(\overline{\Omega})$ with support disjoint from
  $x=\bar x_0$, then $\phi A\psi\in\Psiocuhh^{m,l_2}$.
\item
  If $\phi,\psi\in \CI(\overline{\Omega})$ with disjoint support then
  $\phi A\psi$ has Schwartz kernel which is $C^\infty$ with rapid
  vanishing in $h$ as well as all boundary hypersurfaces of
  $\overline{\Omega}\times\overline{\Omega}$.
  \end{enumerate}
\end{dfn}

Indeed, notice that if $1=\phi_{\scl}+\phi_0+\phi_{\ocul}$ is a
partition of unity with
\begin{equation*}\begin{aligned}
&\supp\phi_{\scl}\cap\{x=0\}=\emptyset,\ \supp\phi_{\ocul}\cap\{x=\bar
x_0\}=\emptyset,\\
&\supp\phi_{\scl}\cap\supp\phi_{\ocul}=\emptyset,\
\supp\phi_0\subset\{0<x<\bar x_0\}
\end{aligned}\end{equation*}
then any two of
$\phi_{\scl},\phi_0,\phi_{\ocul}$ pairwise satisfy one of these
conditions, so e.g.\ $\phi_0A\phi_{\ocul}\in \Psiocuhh^{m,l_2}$, etc.

It is straightforward to check that
$\Psiscocuhh^{\infty,\infty,\infty}$ is a tri-filtered $*$-algebra,
inheriting the properties of the two individual algebras whose
amalgamation it is.

For the X-ray problem then let both $\tilde\chi$ and $\Phi$ be a combination of all the various forms of $\tilde\chi,\Phi$ for
the ingredients. Concretely let
$$
\Phi=F\circ x,
$$
with $F'>0$, $F(x)=-\frac{1}{2x^2}$, for $x$ near $0$,
$F(x)=\frac{1}{\bar x_0-x}$, for $x$ near $\bar x_0$ (but $x<\bar
x_0$), so our exponential weight is $e^{\Phi/h}$. Also let
$$
\tilde\chi=\chi(x^{1/2}\lambda\sqrt{\Phi'}/(h^{1/2}|\alpha|^{1/2})),
$$
with $\chi$ compactly supported, non-negative, identically $1$ near
$0$. Then for $x$ near $0$,
$$
\tilde\chi=\chi(\lambda/(h^{1/2}x |\alpha|^{1/2}))
$$
as considered in the previous section (with the $|\alpha|^{1/2}$
factor irrelevant here, but showing up in the Gaussian being
approximated for ellipticity), while for $x$ near $\bar x_0$,
$$
\tilde\chi=\chi(x^{1/2}\lambda/(h^{1/2}(\bar x_0-x)|\alpha|^{1/2})),
$$
as in the scattering setting of
\cite{Vasy:Semiclassical-X-ray}, with the extra factor of $x^{1/2}$
really should be considered in the context of
$(x\lambda)/(h^{1/2}(\bar x_0-x)|x\alpha|^{1/2})$, in view of the
definition of $\lambda$ and $\alpha$ here, see
\eqref{eq:x-form-sc-flow}, vs.\ in \cite{Vasy:Semiclassical-X-ray}; in
the latter our leading factor of $x$ in \eqref{eq:x-form-sc-flow} would be incorporated into these.
The modified normal operator is then
$$
A=e^{-\Phi/h}L\tilde\chi I e^{\Phi/h}.
$$
A simple combination of the
pseudodifferential computations of the earlier
sections and \cite{Vasy:Semiclassical-X-ray} shows that this operator
is in $h\Psiscocuhh^{-1,-1,-2}(\overline{\Omega})$ provided that there are no conjugate
points on the boundary within distance $\pi/2$ as well as that
geodesics do not have conjugate points to the point of tangency to an
$x$-level set, with the latter following from the former if $\bar x_0$
is sufficiently small. Further, for suitable $\tilde\chi$,
given by approximating a Gaussian $e^{-|.|^2/2}$ on $\RR$ in Schwartz functions by
compactly supported functions $\chi$, the same combination yields ellipticity.
This proves the main theorem:

\begin{thm}\label{thm:main-conic-ops}
On a sufficiently small collar neighborhood of infinity, specified by a level set of $x$ as the artificial boundary, on an asymptotically conic
manifold with no conjugate points within distance $\pi/2$,
the modified normal operator
$e^{-\Phi/h} L\tilde\chi I
  e^{\Phi/h}\in h\Psiscocuhh^{-1,-1,-2}(\overline{\Omega})$ is elliptic in the sense of the standard (differential),
  the 1-cusp (at infinity) and scattering (at the artificial boundary) boundary as well as the semiclassical principal symbols. In
  particular, it is an invertible operator for $h$ sufficiently small.
\end{thm}

As a consequence, we can determine functions from their X-ray
transform {\em without} a support condition.

\begin{crl}\label{cor:main-conic}
The geodesic X-ray $I$, restricted to geodesics that stay in
$x\leq\bar x_0$, is injective on the restrictions to $x\leq\bar x_0$ of sufficiently fast Gaussian
decaying functions.
\end{crl}

As exmpained after Theorem~\ref{thm:main-conic}, this proves Theorem~\ref{thm:main-conic}.

\bibliographystyle{plain}
\bibliography{paperbib}

\end{document}